\documentclass[11pt,twoside, a4paper, english, reqno]{amsart}
\usepackage{graphicx, amsmath, varioref, amscd, amssymb,color, bm, stmaryrd, amsthm, epsfig, color, epic}
\usepackage{fancybox}
\usepackage[pdftex, colorlinks=true]{hyperref}
\usepackage{pdfsync}

\usepackage{eso-pic}
\usepackage{color}
\usepackage{type1cm}
\makeatletter
  \AddToShipoutPicture*{%
    \setlength{\@tempdimb}{.12\paperwidth}%
    \setlength{\@tempdimc}{.5\paperheight}%
    \setlength{\unitlength}{1pt}%
    \put(\strip@pt\@tempdimb,\strip@pt\@tempdimc){%
      \makebox(0,0){\rotatebox{90}{\textcolor[gray]{0.75}{\fontsize{2cm}{2cm}}}}
    }
}

\makeatother

\numberwithin{equation}{section}
\allowdisplaybreaks

\newtheorem{theorem}{Theorem}[section]
\newtheorem{lemma}[theorem]{Lemma}

\newtheorem{proposition}[theorem]{Proposition}
\theoremstyle{definition}
\newtheorem{definition}[theorem]{Definition}

\theoremstyle{remark}

\newcommand{\Div}{\operatorname{div}}

\newcommand{\Grad}{\nabla}
\newcommand{\vr}{\varrho}

\newcommand{\dx}{{\rm d}x}
\newcommand{\dt}{{\rm d}t}
\newcommand{\dxdt}{\dx \ \dt}

\newcommand{\vc}[1]{{\bm{#1}}}

\newcommand{\de}{\delta}
\newcommand{\om}{\omega}

\newcommand{\R}{\mathbb{R}}

\newcommand{\e}{\varepsilon}

\begin{document}

\title[On a tumor growth model]
	  {On a nonlinear model for  tumor growth  in a cellular medium}

\author[Donatelli]{Donatella Donatelli}
\address[Donatelli]{\newline
Departement of Engineering Computer Science and Mathematics\\
University of L'Aquila\\
67100 L'Aquila, Italy.}
\email[]{\href{donatella.donatelli@univaq.it}{donatella.donatelli@univaq.it}}
\urladdr{\href{http://univaq.it/~donatell}{univaq.it/\~{}donatell}}

\author[Trivisa]{Konstantina Trivisa}
\address[Trivisa]{\newline
Department of Mathematics \\ University of Maryland \\ College Park, MD 20742-4015, USA.}
\email[]{\href{http://www.math.umd.edu}{trivisa@math.umd.edu}}
\urladdr{\href{http://www.math.umd.edu/~trivisa}{math.umd.edu/\~{}trivisa}}

\date{\today}



\thanks{}

\maketitle

\begin{abstract}
We investigate the dynamics of a nonlinear model for tumor growth within a cellular medium. In this setting the ``tumor" is viewed as a multiphase flow consisting of  cancerous cells in either proliferating phase or quiescent phase and a collection of cells accounting  for the ``waste" and/or dead cells in the presence of a nutrient.
Here, the tumor is thought of  as a growing continuum  $\Omega$ with boundary $\partial \Omega$  both of which  evolve in time. The key characteristic of the present model is that the total density of cancerous cells is allowed to vary, which is often  the case within  cellular media. We refer the reader to the articles \cite{Enault-2010}, \cite{LiLowengrub-2013} where compressible type tumor growth models are investigated.
Global-in-time weak solutions are obtained  using an  approach based on penalization of the boundary behavior, diffusion,  viscosity   and pressure in the weak formulation, as well as convergence and compactness arguments in the spirit of  Lions \cite{Lions-1998} (see also \cite{Feireisl-book, DT-MixedModel-2013}).
\end{abstract}

\tableofcontents{}

\section{Introduction}\label{S1}
%\subsection{Motivation}
We investigate the dynamics of a nonlinear model for tumor growth within a cellular medium. In this setting the ``tumor" is viewed as a multiphase flow consisting of  cancerous cells in either proliferating phase or quiescent phase and a collection of cells accounting  for the ``waste" or dead cells in the presence of a nutrient (oxygen).
Here, the tumor is thought of  as a growing continuum  $\Omega$ with boundary $\partial \Omega$  both of which  evolve in time. The key characteristic of the present model is that the total density of cancerous cells is allowed to vary. We refer the reader to the articles by Enault \cite{Enault-2010}, where the compressibility effect of the healthy tissue on the invasiveness of a tumor is investigated and to  Li and Lowengrub \cite{LiLowengrub-2013} for references on related models.

\vspace{0.1cm}

This work focuses on major cells such as {\em cancer cells} and {\em dead cells (or waste)} in the presence of a {\em nutrient.} 
Motivated by the experiments of Roda {\em et al.} (2011, 2012) and the mathematical analysis in Bresch et al. \cite{BreschColinGrenierRibbaSaut-2009}, Friedman  \cite{Friedman-2004},  Chen-Friedman \cite{ChenFriedman-2013},  Zhao \cite{Zhao-2010} and Donatelli-Trivisa \cite{DT-MixedModel-2013} our model is based on  the following biological principles: 
%\begin{itemize}
\begin{enumerate}
\item[{\bf [a-1]}] 
Cancer cells are either in a   {\em proliferating phase} or in a  {\em quiescent phase}.
\item[{\bf[a-2]}] 
Proliferating cells die as a result of {\em apoptosis} which is a cell-loss mechanism. 
\item [{\bf[a-3]}] 
Quiescent cells die in part due to {\em apoptosis} but more often due to starvation.
%Living cells undergo {\em mitosis}, a process that takes place in the nucleus of a dividing cell, but for proliferating cells the period of cell cycle is much shorter.
\item[{\bf [a-4]}] 
Cells change from quiescent phase into proliferating phase at a rate which increases with the nutrient level, and they die at a rate which increases as the level of nutrient decreases.
%\item[{\bf [a-5]}]
%The tumor model under consideration refers to avascular tumors (since the nutrient diffuses from outside).

\item[{\bf [a-5]}] 
Proliferating cells, die at a rate which increases as the level of nutrient decreases.
\item[{\bf [a-6]}]  
Proliferating cells become quiescent  at a rate which increases as the nutrient concentration decreases. The proliferation rate increases with the nutrient concentration.
\item[{\bf [a-7]}] 
The total number of cancerous cells can vary as a function of space and time accounting for the case of cancer research investigation within a cellular medium. 
\end{enumerate}
%\end{itemize}  
%The effect of {\em drug treatment} and {\em drug resistance} will be the focus of future investigation \cite{DT-DrugResistance-2014}. 
The system is given by a multi-phase flow model and  the tumor is described as a growing continuum $\Omega(t)$ with boundary $\partial \Omega(t)$, both of which evolve in time.

The tumor region $\Omega_t :=\Omega(t)$ is contained in a fixed domain $B$ and the region $B\setminus \Omega_t$ represents the healthy tissue (see \figurename~\ref{regions}). 
\begin{figure}[htbp] 
%\begin{center}
\centering
\includegraphics[width=5cm]{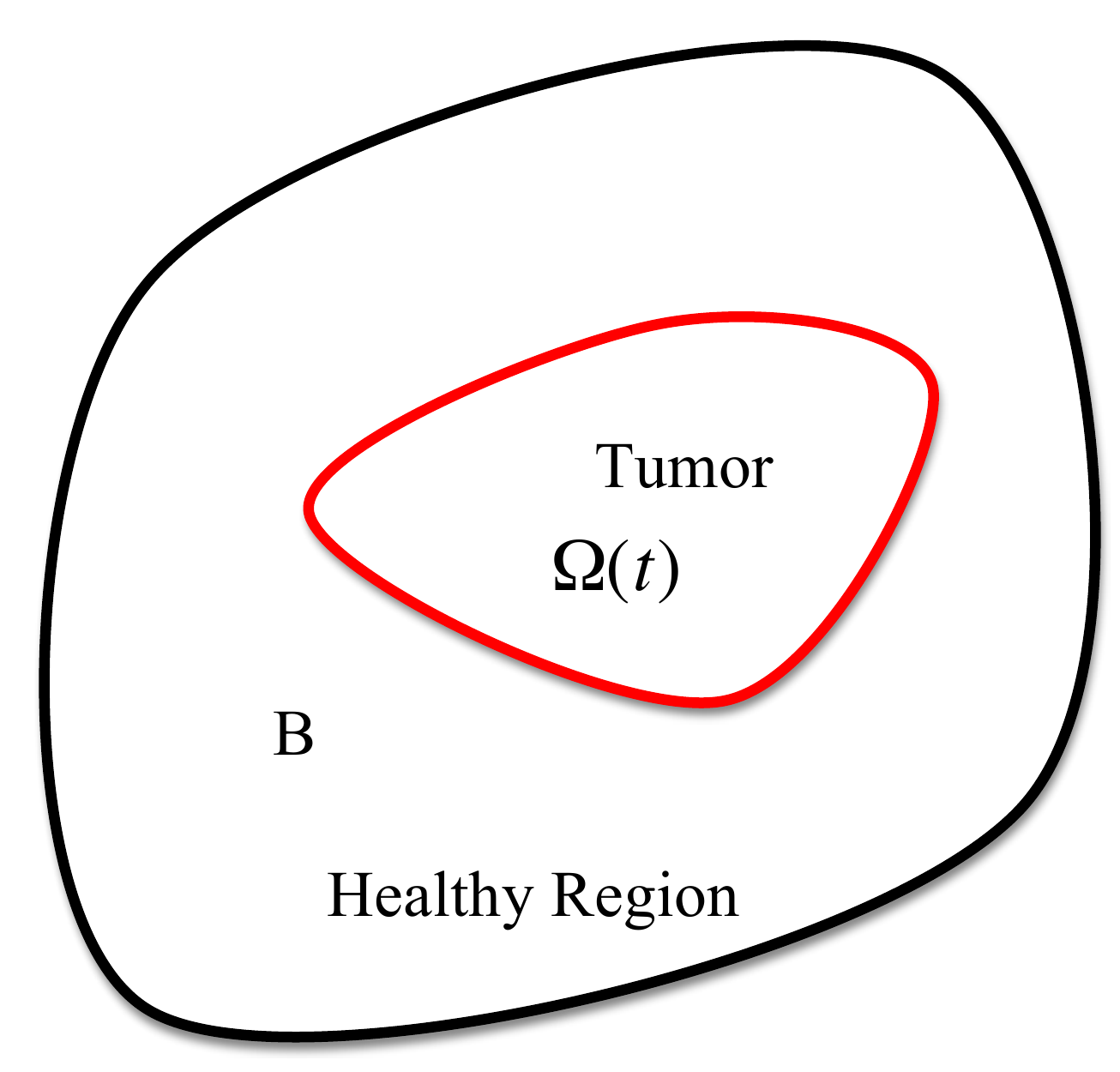}
\caption{Healthy tissue - Tumor regime.} 
\label{regions} 
%\end{center}
\end{figure}

\subsection{Description of the model}
Our aim is to describe the evolution in time of the density (number of cells per unit volume) of few cellular species.  Mathematical models describing continuum cell populations and their evolution typically consider the interactions between the cell number density and one or more chemical species that provide nutrients or influence the cell cycle events of a tumor cell population.
In order to obtain the equations giving the evolution of cellular densities, we use the mass-balance principle for every specie,
\begin{equation}
\partial_t \vr + \Grad \cdot (\vr  {\vc{v}}) = \vc{G}, \nonumber
\end{equation}
where $\vr$ may represent  densities of  cancer cells and dead  cells (waste) within the tumor. The function $G$ includes in general proliferation, apoptosis or clearance of cells, and chemotaxis terms as appropriate.

Cancer cells  are of two types: proliferative cells with density $P(x,t)$ and quiescent cells with density $Q(x,t).$  What is here referred as dead cells with density  $D(x,t)$ includes also what is known in the theory of tumor growth as  the {\em waste} or {\em extra-cellular medium}.  These different populations of cells are in the presence of a nutrient (oxygen) with density $C.$
The rates of change from one phase to another are functions of the nutrient concentration C:
\iffalse
$$ P \to Q \,\, \mbox{at rate}\,\, K_Q (C),$$
$$ Q \to P \,\, \mbox{at rate}\,\, K_P (C), $$
$$ P \to D \,\, \mbox{at rate}\,\, K_A  (C), $$
$$ Q \to D \,\, \mbox{at rate}\,\, K_D (C), $$
\fi
\begin{eqnarray}
&& P \to Q \,\, \mbox{at rate}\,\, K_Q (C), \nonumber \\
&& Q \to P \,\, \mbox{at rate}\,\, K_P (C), \nonumber \\
&& P \to D \,\, \mbox{at rate}\,\, K_A  (C), \nonumber \\
&& Q \to D \,\, \mbox{at rate}\,\, K_D (C),\nonumber
\end{eqnarray}
where $K_Q(C)$ denotes the rate by which proliferating cells change into quiescent cells, 
$K_P(C)$ denotes the rate by which quiescent  cells change into proliferating cells, 
$K_A(C)$ stands for apoptosis, $K_D(C)$ denotes the rate by which quiescent cells die. Finally, dead cells are removed at rate $K_R$ (independent of $C$), and the rate of cell proliferation (new births) is $K_B$
(see \eqref{G}). 
%\smallskip

The total density of the mixture is denoted by $\vr_f$ and is given by\\
\begin{equation}
 \vr= \vr(x,t) =  [P+Q+D](x,t).
\label{density}
\end{equation}
Biologically that means the total density of cancerous cells may vary which is the case when the waste produced by death of any kind does not necessarily remain within the extra cellular medium. 

\subsection{The velocity of the tumor}
The present work considers the mechanical interactions between the various tumor cells in order to see how the mechanical properties of the tumor, and the tissue in which the tumor grows, influence tumor growth. The tumor velocity $\vc{v}$ reflects the continuous motion within the tumor region typically due 
to proliferation and removal of cells and is here given by an alternative to {\em Darcy's Law} known in the porous medium literature as  {\em Forchheimer's equation} 
\begin{equation}
\partial_t (\vr \vc{v}) + \Div (\vr \vc{v} \otimes \vc{v}) +  \Grad \sigma(P, Q, D) = \mu \Delta \vc{v}-\frac{\mu}{K} \vc{v}, \label{pressure2}
\end{equation}
\smallskip
where $\mu$  is a positive constant describing the viscous-like properties of tumor cells, $K$ denotes the permeability, and $\sigma$ denotes the pressure given by, 
\begin{equation}
\sigma(P,Q,D) = P^m+ Q^m+ D^m,\,\,\, m>\frac{3}{2}. \label{pressure-com}
\end{equation}

 %Equation \eqref{pressure2} can be interpreted in two different ways. The first interpretation is as an over damped force balance, in which the force on a cell maintains  its velocity rather than accelerates it. The second 

Equation \eqref{pressure2} can be interpreted as follows.  The tumor tissue is in this setting ``fluid-like" and  the tumor cells  ``flow" through the cellular medium  like a fluid flows through a porous medium, obeying Forchheimer's law. More specifically,

\begin{itemize}
\item The second term on the right-hand side of \eqref{pressure2} is the usual Darcy's law, and in the  present setting  results from  the {\em friction} of the tumor cells with the extracellular matrix. 

\item  The first term of the right-hand side,  is a  dissipative force density and  results from the {\em internal cell friction} due to cell volume changes.

\item The present work takes into consideration the variation of the cell densities and velocity variations within the cellular medium which are  often considered negligible  in biological settings but are rather substantial in some regimes of interest. 
\end{itemize}

\subsection{Equations for the populations  of  cells}
The mass conservation laws for the densities of the proliferative cells $P,$ quiescent cells $Q$ and dead cells $D$ take the following form:

\begin{eqnarray}
&& \partial_t P + \Div (P \vc{v}) = \vc{G_P}, \label{dP}\\
&& \partial_t Q  + \Div (Q \vc{v}) = \vc{G_Q}, \label{dQ}\\
&& \partial_t D + \Div (D \vc{v}) = \vc{G_D}. \label{dD}
\end{eqnarray}

Following Friedman \cite{Friedman-2004}, the source terms ${\bf \{G_P, G_Q, G_D\}}$ are of the form
\begin{equation}
\begin{split}
& \vc{G_P} =  \left(K_B C - K_Q (\bar C- C) - K_A(\bar C - C)\right) P + K_P C Q\\
& \vc{G_Q} = K_Q (\bar C- C)P - \left(K_P C + K_D(\bar C - C)\right) Q\\
& \vc{G_D} =  K_A (\bar C- C)P  + K_D (\bar C-C) Q - K_R D,
\end{split}\nonumber
\end{equation}
where $C$ is the nutrient concentration. 

Without loss of generality (cf.\ \cite{BreschColinGrenierRibbaSaut-2009} and the references therein) we consider here ${\bf \{G_P, G_Q, G_D\}}$  in the following simplified version:\\
\begin{equation}
\begin{split} \label{G}
& \vc{G_P} =  \left(K_B C - K_Q (\bar C- C)- K_A(\bar C - C)\right) P \\
&  \vc{G_Q} =   - \left(K_P C + K_D(\bar C - C)\right) Q\\
&  \vc{G_D} =  - K_R D. 
\end{split}
\end{equation}

\subsection{A linear diffusion equation for the nutrient concentration}
Unlike tumor cells the density of the nutrient (oxygen) obeys a linear diffusion equation.
It is well known that tumor cells consume nutrients, which  diffuse into the tumor tissue
from the surrounding tissue (cf. \cite{RCM-2007}). The nutrient concentration $C$ satisfies a linear diffusion equation of the form\\
\begin{equation}
\frac{\partial C}{\partial t} = D_1 \Delta C -  \left(K_1 K_P CP + K_2 K_Q(\bar{C}-C)Q\right)C. \nonumber
\end{equation}
and for simplicity, we take (see \cite{Friedman-2004}) , 
\begin{equation}
\frac{\partial C}{\partial t} = \nu \Delta C -  K_{C} C, \label{dC2}
\end{equation}
where $\nu>0$ is a diffusion coefficient and without loss of generality we consider $K_{C}=1$.  We refer the reader also to \cite{BreschColinGrenierRibbaSaut-2009}, where a stationary (elliptic) equation  was introduced referring to the case  when the diffusion time-scale of the oxygen is much lower than the time scale of cellular division.
The present setting takes into consideration the dynamic aspects of the diffusion.

\smallskip
Adding \eqref{dP}-\eqref{dD} and taking into consideration \eqref{density} and \eqref{G} we arrive at the following transport relation for the evolution of the total density of the mixture
 \begin{align}
\partial_t \vr + \Div (\vr \vc{v}) =& \vc{G_P} + \vc{G_Q}+\vc{G_D}\nonumber\\
=& (K_{A}+K_{B}+K_{Q})CP-(K_{A}+K_{Q})\bar{C}P\nonumber\\
&-K_{D}\bar{C}Q+(K_{D}-{K_{P}})CQ-K_{R}D.
\label{divcon}
\end{align}
Our aim is to study the system \eqref{pressure2}-\eqref{dC2} in a spatial domain $\Omega_t$, with a boundary $\Gamma =\partial\Omega_t$ varying in time. 

\subsection{Boundary behavior}
The boundary of the domain $\Omega_t$ occupied by the tumor is described by means of a given velocity $\vc{V}(t, \vc{x}),$ where $t \ge 0$ and $\vc{x} \in \R^3.$ More precisely, assuming $\vc{V}$ is regular, we solve the associated system of differential equations
\begin{equation}
\frac{d}{dt} \vc{X}(t, \vc{x}) = \vc{V}(t, \vc{X}), \,\, t > 0, \,\, \vc{X}(0,\vc{x}) = \vc{x}, \nonumber
\end{equation}
and set
\begin{equation}
\begin{cases}
\!\!\! \! & \Omega_{\tau} = \vc{X}(\tau, \Omega_0), \,\, \mbox{where}\,\, \Omega_0 \subset \R^3\,\, \mbox{is a given domain,}\\
\!\! \!\! & \Gamma_{\tau} = \partial \Omega_{\tau}, \,\, \mbox{and} \,\, Q_{\tau} = \left\{(t,x) | t \in (0,\tau), x\in  \Omega_{\tau}\right\}.
\end{cases}\nonumber
\end{equation}
\smallskip

 The model is closed by giving boundary conditions on the (moving) tumor boundary $\Gamma_{\tau}.$ 
More precisely, we assume that the boundary $\Gamma_{\tau}$ is impermeable,  meaning
\begin{equation}
 (\vc{v} - \vc{V}) \cdot \vc{n}|_{\Gamma_{\tau}} = 0, \,\, \mbox{for any}\,\, \tau \ge 0. \label{BC1}
 \end{equation}
In addition, for {\em viscous} fluids, Navier proposed the boundary condition of the form
\begin{equation}
[\mathbb{S} \vc{n}]_{\mbox{tan}}|_{\Gamma_{\tau}} = 0, \label{BC2}
\end{equation}
with $\mathbb{S}$ denoting the viscous stress tensor which in this context is assumed to be determined through Newton's rheological law
$$
\mathbb{S} = \mu \Big( \Grad \vc{v} + \Grad^{\perp} \vc{v} - {2 \over 3} \Div
\vc{v} \mathbb{I} \Big) + \xi \Div \vc{v} \mathbb{I}.
$$
The constants  $\mu> 0$, $\xi \geq 0$ are respectively the shear and bulk
viscosity coefficients. Condition \eqref{BC2} namely says that the tangential component of the normal 
viscous stress vanishes on $\Gamma_{\tau}.$
The concentration of the nutrient on the boundary satisfies the condition:
\begin{equation}
C|_{\Gamma_t} = 0  \label{BC3}.
\end{equation} 
Finally, the problem \eqref{dP}-\eqref{BC3} is supplemented by the initial conditions
\begin{equation}
\begin{split}
P(0, \cdot) = P_0, \quad Q(0, \cdot) = Q_0, \quad D(0, \cdot) = D_0, \\
C(0, \cdot) = C_{0}\leq\bar{C}, \quad \vc{v}(0,\cdot)=\vc{v}_{0}\ \qquad \quad  \text{in $ \Omega_0$.}
\end{split}
 \label{IC}
\end{equation}
Our main goal is to show the existence of global in time weak solutions to  \eqref{pressure2}-\eqref{IC} for any finite energy  initial data. Related works on the mathematical analysis of cancer  models have been presented by Friedman {\em et al.}  \cite{Friedman-2004}, \cite{ChenFriedman-2013}  who established the local existence of  radial symmetric smooth solutions to a related model. 
 The analysis in \cite{Zhao-2010} treated a parabolic-hyperbolic free boundary problem and provided a unique global solution in the radially symmetric case. In the forth mentioned articles  the tumor tissue is assumed to be a fluid flowing a porous medium and the velocity field is  determined by  Darcy's Law
$$\vc{v} = - \Grad_x \sigma \,\, \mbox{in} \,\, \Omega(t).$$
In \cite{DT-MixedModel-2013} Donatelli and Trivisa obtained the global existence of weak solutions to a nonlinear model for tumor growth in a general domain $\Omega_t \subset \R^3$ without  any kind of symmetry assumptions. The article \cite{DT-MixedModel-2013} treated the tumor tissue as a fluid flowing in a  porous medium with the velocity field given by Brinkman's equation 
$$\Grad \sigma = - \frac{\mu}{K} \vc{v} + \mu \Delta \vc{v},$$
and focused on the case of constant  total density of  cancerous cells. In \cite{DT-MixedModelDrug-2014}, the same authors treat a related nonlinear model and discuss the effect of drug application on tumor growth.
\smallskip

The main contribution of the present article to the existing theory can be characterized
as follows:

\begin{itemize}
\item 
The present work treats the tumor as a mixture with a {\em variable} total density of cancerous cells. In accordance, the velocity of the tumor  verifies an extension of Darcy's law known as Forchheimer's equation obtained by analogy to the Navier-Stokes equation. The global existence of weak solutions within a moving domain in $\R^3$ is obtained without assuming any kind of symmetry. For related works involving compressible-type models for the investigation of tumor growth models we refer the reader to Enault \cite{Enault-2010}, Li and Lowengrub \cite{LiLowengrub-2013} and the references therein.

\item {\color{black} The framework presented here relies on biologically grounded principles [{\bf a-1}]-[{\bf a-7}], which are motivated by experiments performed by Roda {\em et al.} \cite{Roda-etal-2011} \cite{Roda-etal-2012A}, \cite{Roda-etal-2012}  and provide  a  description of the dynamics of the population of cells within the tumor.}
\end{itemize}

\par\smallskip

We establish the global existence of  weak solutions to \eqref{pressure2}-\eqref{IC} on time dependent domains, supplemented with slip boundary conditions. The existence theory for the barotropic Navier-Stokes
system on {\em fixed} spatial domains in the framework of weak solutions was developed in the seminal work of Lions \cite{Lions-1998}. 
%The investigation of {\em incompressible} fluids in time dependent domains started with a seminal paper of Ladyzhenskaja \cite{Ladyzhenskaja}, see also \cite{FeireislNS-2011, FeireislKNNS-2013} for related results in that direction.

\par\smallskip

The main ingredients of our approach can be formulated as follows:
\begin{itemize}
\iffalse
\item {\color{blue} Biologically grounded hypotheses regarding the phase of the cells within the tumor and the processes that they undergo at different stages. These
hypotheses  are motivated by experiments performed by Roda {\em et al.} \cite{Roda-etal-2011} \cite{Roda-etal-2012}  and provide  a description of the cells within the tumor and their dynamic properties.}
\fi
\item  In the construction of a suitable approximating scheme the {\em penalizations} of the boundary behavior, diffusion and viscosity are introduced  in the weak formulation. A penalty approach to slip conditions for {\em stationary incompressible flow} was proposed by Stokes and Carey \cite{StokesCarey-2011} (see also \cite{DT-MixedModel-2013, FeireislKNNS-2013}). In the present setting, the variational (weak) formulation of the Forchheimer's  equation is supplemented by a singular forcing term
\begin{equation} 
\frac{1}{\varepsilon} \int_{\Gamma_t} (\vc{v}-\vc{V}) \cdot {\bf n} \vc{\varphi} \cdot \vc{n} dS_x,\,\,\, \varepsilon>0\,\, \mbox{small}, \label{penalty}
\end{equation}
penalizing the normal component of the velocity on the boundary of the tumor domain.

\item In addition to  \eqref{penalty}, we introduce a {\em variable} shear viscosity coefficient $\mu = \mu_{\omega},$ as well as a {\em variable} diffusion $\nu=\nu_{\omega}$ with $\mu_{\omega}, \nu_{\omega}$  vanishing outside the tumor domain and remaining positive within the tumor domain, to accommodate the time-dependent nature of the boundary.

\item In constructing the approximating problem we employ a number of regularizations/penalizations: $\eta, \e, \de, \om$.
Keeping $\eta, \e, \de, \om$  fixed, we solve the modified problem in a (bounded) reference domain $B \subset \mathbb{R}^3$ chosen in such way that 
$$\bar{\Omega}_{\tau} \subset B \,\, \mbox{for any}\,\, \tau \ge 0.$$  
Letting $\eta \to 0$ we obtain the solution $(P,Q,D)_{\de, \om,\e}$ within the fixed reference domain. 
\item We take the initial densities $(P_0, Q_0, D_0)$ vanishing outside $\Omega_0,$  and letting the penalization  $\varepsilon \to 0$ for fixed $\omega > 0 $ we obtain a ``two-phase" model consisting of the {\em tumor region}  and the {\em healthy tissue} separated by impermeable boundary.  We show that the densities vanish  in part of the reference domain, specifically on $((0,T) \times B) \setminus Q_T.$ 

\item We let first the penalization $\varepsilon$ vanish and next we perform the limit $\omega \to 0$ and $\delta \to 0.$

\item The slip boundary conditions considered  here are suitable in the context of moving domains and biologically relevant as confirmed  by experimental evidence. 
\end{itemize}

\par\smallskip

The paper is organized as follows: Section \ref{S1} presents the motivation, modeling  and   introduces the necessary preliminary material. Section \ref{S2}  provides a weak formulation of the problem and states the main result. Section \ref{S3}  is devoted to the penalization problem and to the construction  of a suitable approximating scheme. In Section \ref{S4}  we  present the modified energy inequality and collect all the uniform bounds satisfied by the solution of the approximating scheme. In Section \ref{S5},  we derive essential pressure estimates. In Section \ref{S6}  the singular limits for $\varepsilon \to 0$ is performed. The key ingredient at this step is the establishment of the strong convergence of the density, which is obtained, in analogy to the theory of compressible Navier-Stokes equation, by establishing the weak continuity of the effective viscous pressure. Subsequently it is proven that in fact   the proliferating, quiescent, dead cells and the nutrient are vanishing in the healthy tissue. In Sections \ref{S7} and \ref{S8}  the singular limits  $ \omega \to 0$ and $\delta \to 0$ are performed successively. 

\section{Weak formulation and main results}\label{S2}
%\subsection{Weak solutions}
\begin{definition}\label{D2.1}
 We say that $(P, Q, D, \vc{v}, C)$ is a weak solution of problem
(\ref{dP})-(\ref{IC}) supplemented with boundary data satisfying
(\ref{BC1})-(\ref{BC3})  and initial data $(P_0,  Q_0, D_0, \vc{v}_{0}, C_0)$
satisfying (\ref{IC}) provided that the following hold:
\vspace{0.1in}

$\bullet$ $(P,Q, D) \ge 0$ represents a weak solution of \eqref{dP}-\eqref{dQ}-\eqref{dD} on $[0,T]\times\Omega_{\tau}$, i.e., for any test function $\varphi \in C^{\infty}_c ([0,T]\times \mathbb{R}^3), T>0$, for any $\tau\in[0,T]$ the  following integral relations hold
\smallskip

\begin{equation} 
\left.
\begin{array}{l}
\displaystyle{\int_{\Omega_{\tau}}  P \varphi(\tau,\cdot) \, dx  - \int_{\Omega_0}  P_0 \varphi(0,\cdot)dx = }\\
\hspace{1.5cm}\displaystyle{\int_0^{\tau} \!\!\int_{\Omega_t} \left( P \partial_t \varphi + P \vc{v} \cdot \Grad_x \varphi + \vc{G_P}  \varphi(t, \cdot) \right) dx dt}, \\ \\
\displaystyle{\int_{\Omega_{\tau}}  Q \varphi(\tau,\cdot) \, dx - \int_{\Omega_0}  Q_0 \varphi(0,\cdot)dx}  = 
 \\
\hspace{1.5cm}\displaystyle{\int_0^{\tau} \!\!\int_{\Omega_t} \left( Q \partial_t \varphi + P \vc{v} \cdot \Grad_x \varphi + \vc{G_Q}  \varphi(t, \cdot) \right) dx dt,}\\ \\
 \displaystyle{\int_{\Omega_{\tau}} D \varphi(\tau,\cdot) \, dx  - \int_{\Omega_0}  D_0 \varphi(0,\cdot)dx  =} \\
 \hspace{1.5cm}\displaystyle{\int_0^{\tau} \!\!\int_{\Omega_t} \left( D \partial_t \varphi + D \vc{v} \cdot \Grad_x \varphi + \vc{G_D}  \varphi(t, \cdot) \right) dx dt}.
\end{array}
\right\}
%\tag {\bf{I}}
 \label{w-Da}
\end{equation}
%In the sequel we refer to these relations as {\text{\bf (I)}}.
In particular, 
$$P \in L^{\infty}([0,T]; L^m(\Omega_{\tau})), \,\, Q \in L^{\infty}([0,T]; L^m( \Omega_{\tau})), \,\,D \in L^{\infty}([0,T]; L^m( \Omega_{\tau})). $$ 
We remark that in  the weak formulation, it is convenient that the equations \eqref{dP}-\eqref{dD} hold in the whole space $\mathbb{R}^3$ provided that the densities $P,Q,D$ are extended to be zero outside the tumor domain.
\smallskip

%\item 
$\bullet$ Forchheimer's equation \eqref{pressure2} holds in the sense of distributions, i.e., for any test function 
$\vc{\varphi} \in C^{\infty}_c(\mathbb{R}^3; \mathbb{R}^3)$ satisfying 
$$ \vc{\varphi} \cdot  \vc{n}|_{\Gamma_{\tau}} = 0\,\, \mbox{for any}\,\, \tau \in [0,T],$$ 
the following integral relation holds
\begin{equation}
\label{w-pressure2}
\begin{split}
 \int_{\Omega_{\tau}}\vr\vc{v}\cdot\vc{\varphi}(\tau,\cdot)dx&-\int_{\Omega_{0}}{ (\vr\vc{v})_0 \cdot \vc{\varphi} (0, \cdot)dx}\\
= \int_{0}^{\tau}\!\! \int_{\Omega_{\tau}} &\Big( \vr \vc{v} \cdot\partial_t \vc{\varphi} + \vr \vc{v} \otimes \vc{v} : \Grad_x \vc{\varphi}\\ 
&+ \sigma(P,Q,D) \Div \vc{\varphi}-\mu
 \Grad_x \vc{v}:\Grad_x \vc{\varphi} - \frac{\mu}{K} \vc{v} \vc{\varphi}\Big)dxdt 
 \end{split}
 \end{equation}
 The impermeability boundary condition \eqref{BC1} is satisfied in the sense of traces, namely

$$\vc{v} \in L^{2}([0,T];W^{1,2}(\mathbb{R}^3;\mathbb{R}^3)),$$
and
$$ (\vc{v - V}) \cdot \vc{n}(\tau, \cdot)|_{\Gamma_{\tau}}=0\,\, \mbox{for a.a.}\,\, \tau \in [0,T].$$

\smallskip

%\item 
$\bullet$ $C \geq 0$ is a weak solution of \eqref{dC2}, i.e.,  for any test function $\varphi \in C^{\infty}_c ([0,T)\times \mathbb{R}^3), T>0$, for any $\tau\in[0,T]$ 
the  following integral relations hold
\begin{equation}
\begin{split}
\int_{\Omega_{\tau}}&  C \varphi(\tau,\cdot) \, dx - \int_{\Omega_0}  C_0 \varphi(0,\cdot)dx \\
& = \int_0^{\tau} \!\!\int_{\Omega_{t}} \left( C \partial_t \varphi  - \nu \Grad_x C\cdot \Grad_x \varphi -  C  \varphi\right)dx dt. 
\end{split}
\label{wC}
\end{equation}
\smallskip

%$\bullet$ The energy of the system is finite and bounded by the initial energy. ???
%\begin{equation}
%\nonumber \intTime\intOmega\mu\vert\Grad\textbf{u}\vert^2+\lambda\vert\Div\textbf{u}\vert^2+\vert 2\Grad\sqrt{\eta}+\sqrt{\eta}\Grad\Phi\vert^2 \dx \dt\leq\mathcal{F}(\vr_0,\textbf{u}_0,\eta_0)
%\end{equation}
%\end{itemize}
\end{definition}

The main result of the article now follows. 

\begin{theorem}\label{T2.2}
Let $\Omega_0 \subset \mathbb{R}^3$ be a bounded domain of class $C^{2+\nu}$  and let 
$$\vc{V} \in C^1([0,T]; C^3_c(\mathbb{R}^3; \mathbb{R}^3))$$
 be given. Let the initial data satisfy
 $$P_0 \in L^{m}(\mathbb{R}^3), \,\, Q_0 \in L^{m}(\mathbb{R}^3),\,\, D_0 \in L^{m}(\mathbb{R}^3),\,\, C_0 \in L^{m}(\mathbb{R}^3),$$ 
$$ (P_0, Q_0, D_0, C_{0}) \ge 0, \,\,\,  (P_0, Q_0, D_0, C_{0}) \not\equiv 0, \,\,\, (P_0, Q_0, D_0, C_{0})|_{\mathbb{R}^3 \setminus \Omega_0} =0 $$
for a certain $\displaystyle{m > \frac{3}{2}}.$
Denoting by $\vr$ the total density of cells, namely
$$\vr(x,t) = [P + Q + D](x,t)$$
we require that 
$$(\vr \vc{v})_0 = 0 \,\, a.a.\, on \,\,  \{ \vr_0=0\},\,\, \int_{\Omega_0} \frac{1}{\vr_0} |(\vr \vc{v})_0|^2 \dx < \infty.$$ 
%and pressure $\sigma$ such that \eqref{pressure}, \eqref{pressure2} are satisfied.
Then the problem \eqref{dP}-\eqref{dC2} with initial data \eqref{IC}  and boundary data \eqref{BC1}, \eqref{BC2} and  \eqref{BC3}  admits a weak solution in the sense specified in Definition \ref{D2.1}.
\end{theorem}

\section{Approximating Scheme }\label{S3}
In the heart of the approximating procedure presented here lie the so-called  {\em generalized penalty methods}, which entail treating the boundary condition as a weakly enforced constraint. This approach has appeared to be suitable for treating partial slip, free surface, contact and related boundary conditions in viscous flow analysis and simulations. In incompressible  viscous flow modeling such approach provides penalty enforcement of the incompressibility constraint on the velocity field \cite{CareyKrishnan-1982}, \cite{CareyKrishnan-1984},\cite{CareyKrishnan-1985}.  

The form of boundary penalty approximation introduced here has its origin  in  Courant  \cite{Courant-1956}. A penalty approach to slip conditions for stationary incompressible fluids was proposed
by Stokes and Carey \cite{StokesCarey-2011}. Compressible fluid flows in time dependent domains, supplemented with the no-slip boundary conditions, were examined in \cite{FeireislNS-2011} by means of Brinkman's penalization method and in \cite{FeireislKNNS-2013} treating a slip boundary condition. A penalty approach to the analysis of a tumor growth model was presented in \cite{DT-MixedModel-2013} treating the case of a mixed-type tumor growth model.

It is clear   that  applying a penalization method to the slip boundary conditions is much more delicate than the treatment of no-slip boundary conditions. Indeed, in the case of slip boundary conditions we have information only  for  the normal component  $\vc{v} \cdot {\bf n}$  outside $\Omega_{\tau}.$

The central component in the construction of a suitable approximating scheme is the addition of  a singular forcing term
\begin{equation*}
\frac{1}{\varepsilon}  \int_{\Gamma_t} (\vc{v}-\vc{V}) \cdot {\bf n} \vc{\varphi} \cdot \vc{n} dS_x ,\,\,\, \varepsilon>0\,\, \mbox{small}, 
%\label{penalty}
\end{equation*}
penalizing the normal component of the velocity on the boundary of the tumor domain
in the variational formulation of Forchheimer's equation.
\subsection{Penalization}

As typical in time dependent  regimes the penalization can be applied to the interior of a fixed reference domains. In that way we obtain at the limit a two-phase model consisting of the tumor region  $\Omega_\tau$ and a healthy tissue $B\setminus \Omega_{\tau}$ separated by an impermeable interface
$\Gamma_{\tau}.$  As a result an extra stress is produced acting on the fluid by its complementary part outside $\Omega_{\tau}.$

We  choose  $R>0$ such that 
\begin{equation}
\vc{V}|_{[0,T] \times \{|\vc{x}| > R\} }=0, \,\,\, \bar{\Omega}_0 \subset \{|\vc{x}| <R\}\label{RefDom}
\end{equation}
and we take as the reference fixed domain 
$$B = \{|\vc{x}| < 2R\}.$$

In order to eliminate this extra stresses we introduce  a  four level penalization scheme, which relies on the parameters $\eta$ which plays the role of the artificial viscosity in the equations \eqref{dP},\eqref{dQ}, \eqref{dD},  $\varepsilon$ which accounts for the penalization of the boundary behavior, $\omega$ which introduces penalization of the viscosity and diffusion parameters and $\delta$ which represents the artificial pressure and will be instrumental in the establishment of the pressure estimates and in the proof of the strong convergence of the densities. In the description of the approximating scheme below we mention the parameter $\eta$ only briefly and the details of the limit $\eta \to 0$  which have been presented in a series of articles are omitted. We refer the reader to \cite{Donatelli-Trivisa-2006, Feireisl-book} for details.
    
Our approximating scheme relies on:
\begin{enumerate}
\item[{\bf 1.}] A variable shear viscosity coefficient $\mu = \mu_{\omega}(t, \vc{x}),$ where $\mu=\mu_{\omega}$ remains strictly positive in $Q_{T}$ but vanishes in $Q_{T}^{c}$ as $\omega \to 0$, namely $ \mu_{\omega}$
is taken such that 
$$ \mu_{\omega} \in C^{\infty}_c \left([0,T] \times \mathbb{R}^3\right), \,\,\, 0<{\underline{{\mu}}}_{\omega}\le \mu_{\omega}(t,\vc{x}) \le \mu\,\, \mbox{in}\,\, [0,T]\times B,$$
\begin{equation}
\mu_{\omega}=
\begin{cases}
\mu={\mathrm const}>0 & \text{in $Q_{T}$}\\
\mu_{\omega}\to 0 & \text{a.e. in $((0,T)\times B)\backslash Q_{T}$}
\end{cases}\nonumber
 %\label{visc-p}
\end{equation}
and a variable diffusion coefficient of the nutrient $\nu = \nu_{\omega}(t, \vc{x}),$ where $\nu=\nu_{\omega}$ remains strictly positive in $Q_{T}$ but vanishes in $Q_{T}^{c}$ as $\omega \to 0$, namely $ \nu_{\omega}$
is taken such that 
$$ \nu_{\omega} \in C^{\infty}_c \left([0,T] \times \mathbb{R}^3\right), \,\,\, 0<{\underline{{\nu}}}_{\omega}\le \nu_{\omega}(t,\vc{x}) \le \nu\,\, \mbox{in}\,\, [0,T]\times B,$$
\begin{equation}
\nu_{\omega}=
\begin{cases}
\nu={\mathrm const}>0 & \text{in $Q_{T}$}\\
\nu_{\omega}\to 0 & \text{a.e. in $((0,T)\times B)\backslash Q_{T}.$}
\end{cases}\nonumber
 %\label{diff-p}
\end{equation}
\item[{\bf 2.}] An artificial pressure is introduced 
$$\sigma_{\delta}(P,Q,D)  = \sigma(P,Q,D) + \delta (P^{\beta} + Q^{\beta} + D^{\beta}). $$
where $\delta>0$ and $\beta \geq 2$.

\item[{\bf 3.}] 
We modify the initial data for $P$, $Q$, $D$, $C$ and $\vr \vc{v}$ so that the following set of relations hold
\begin{equation}
\left.
\begin{array}{r}
P_{0}= P_{0,\delta, \omega,\e}= P_{0,\delta,\omega}=P_{0,\delta}, \ Q_{0}=Q_{0,\delta, \omega,\e} = Q_{0,\delta,\omega}=Q_{0,\delta} \\ \\
D_{0}=D_{0,\delta, \omega,\e, \delta}=D_{0, \delta,\omega}= D_{0,\delta}\ C_{0}=C_{0,\delta, \omega,\e}=C_{0,\delta,\omega}=C_{0,\delta}\\ \\
 P_{0,\delta,\omega,\e}\geq 0, \quad  Q_{0,\delta, \omega,\e}\geq 0\quad  D_{0,\delta,\omega,\e}\geq 0\quad  C_{0,\delta,\omega,\e }\geq 0\\ \\
 P_{0,\delta,\omega,\e}\not \equiv 0, \ Q_{0,\delta,\omega,\e}\not \equiv  0,  \ D_{0,\delta,\omega,\e}\not \equiv  0,\ C_{0,\delta,\omega,\e}\not\equiv 0\\ \\
 P_{0,\delta,\omega,\e},\ Q_{0,\delta,\omega,\e},\  D_{0,\delta, \omega,\e},\ C_{0,\delta,\omega,\e}|_{\R^{3}\backslash\Omega_{0}}=0,\\ \\
 \displaystyle{\int_{B}(P_{0,\delta, \omega,\e}^{m}+\delta P_{0,\delta, \omega,\e})dx\leq c,\   \int_{B}(Q_{0,\delta, \omega,\e}^{m}+\delta Q_{0,\delta, \omega,\e}^{\delta})dx}\leq c\\ \\
 \displaystyle{\int_{B}(D_{0,\delta, \omega,\e}^{m}+\delta D_{0,\delta, \omega,\e})dx\leq c}\\ \\
 (\vr \vc{v})_0 =  (\vr \vc{v})_{0,\delta,\omega,\e}=(\vr \vc{v})_{0,\delta,\omega}=(\vr \vc{v})_{0,\delta}=0 \,\, a.a.\, on \,\,  \{ \vr_0=0\},\\
  \displaystyle{\int_{\Omega_0} \frac{1}{\vr_0} |(\vr \vc{v})_0|^2 dx < \infty.}
 \end{array}
\right \}
\label{IP}
\end{equation}
 \item[{\bf 4.}] Keeping $\eta, \varepsilon, \delta, \omega >0$ fixed, we solve the modified problem in the fixed reference domain $B \subset  \R^3$  chosen  as in \eqref{RefDom} with
$\bar{\Omega}_{\tau} \subset B, \, \tau \ge 0.$  The approach used at this level employs  the Faedo-Galerkin method, which  involves  replacing the regularized Forchheimer's equation 
by a system  of integral equations, with $P, Q, D$ 
 being exact solutions of the {\em regularized} \eqref{dP}, \eqref{dQ} and 
\eqref{dD} (involving the parameter $\eta$ mentioned above which appears as an  artificial viscosity).  Given $\eta, \e, \de, \om$ positive fixed, these parabolic equations can be solved with the aid of a suitable fixed point argument providing the approximate cell densities. Next, using the integral form of the regularized Forchheimer's equation and performing a fixed point argument one obtains the  approximate velocity. By taking the limit as the dimension of the basis used in the Faedo-Galerkin approximation  tends to $\infty$ we obtain  the solution $(P_{\de,\om, \e, \eta}, Q_{\de,\om,\e,\eta}, D_{\de,\om, \e, \eta}, \vc{v}_{\de,\om, \e,\eta})$ within  the fixed reference domain $B.$ Next, we let $\eta \to 0$ following the line of arguments presented in \cite{Donatelli-Trivisa-2006, Feireisl-book} establishing the existence of the solution $(P_{\de,\om,\e}, Q_{\de,\om,\e}, D_{\de,\om,\e}, \vc{v}_{\de,\om,\e})$ within $B.$

\item[{\bf 5.}] 
Letting $\e \to 0$ we obtain a ``two-phase'' system, where the density vanishes in the healthy tissue of the reference domain.
Next, we perform the limit $\omega \to 0,$ where the extra stresses disappear  in the limit system. The desired conclusion follows from the final limit process $\delta \to 0.$ 
\end{enumerate} 

The weak formulation of the penalized problem reads:
%\iffalse
%\begin{equation}
%\int_B \vr \varphi(\tau,\cdot) \, dx - \int_B \vr_0 \varphi(0,\cdot)dx = \int_0^{\tau} \int_B (\vr \partial_t \varphi + \vr \vc{u} \cdot \Grad_x \varphi) dx dt +   \int_0^{\tau} \int_B \vc{G}_{\vr}  \varphi(t, \cdot) dx dt\label{w-cont}.
%\end{equation}
%
%\begin{equation}
%\int_B P \varphi(\tau,\cdot) \, dx - \int_B P_0 \varphi(0,\cdot)dx = \int_0^{\tau} \int_B \left( P \partial_t \varphi + P \vc{v} \cdot \Grad_x \varphi + \vc{G_P}  \varphi(t, \cdot) \right) dx dt \label{w-P}.
%\end{equation}
%\begin{equation}
%\int_B Q \varphi(\tau,\cdot) \, dx - \int_B Q_0 \varphi(0,\cdot)dx = \int_0^{\tau} \int_B \left( Q \partial_t \varphi + P \vc{v} \cdot \Grad_x \varphi + \vc{G_Q}  \varphi(t, \cdot) \right) dx dt \label{w-Q}.
%\end{equation}
%\begin{equation}
%\int_B D \varphi(\tau,\cdot) \, dx - \int_B D_0 \varphi(0,\cdot)dx = \int_0^{\tau} \int_B \left( D \partial_t \varphi + D \vc{v} \cdot \Grad_x \varphi + \vc{G_D}  \varphi(t, \cdot) \right) dx dt \label{w-D}.
%\end{equation}
%\fi
%\smallskip

\begin{itemize}
\item The integral relations \eqref{w-Da} in Definition \eqref{D2.1} hold true
for any $\tau \in [0,T]$ and $\vc{x} \in B$ and any test function $\varphi \in C_c^{\infty}([0,T] \times \mathbb{R}^3),$ and for 
$(\vc{G_{P_{\delta,\omega, \e}}}, \vc{G_{Q_{\delta,\omega, \e}}}, \vc{G_{D_{\delta,\omega, \e}}})$ given in \eqref{G}, namely
\begin{equation} 
\left.
\begin{array}{l}
 \hspace{0.5cm}\displaystyle{\int_{B}  P_{\delta,\omega,\e} \varphi(\tau,\cdot) \, dx - \int_{\Omega_0}  P_0 \varphi(0,\cdot)dx=}\\
 \hspace{0.6cm}\displaystyle{
 \int_0^{\tau} \!\!\int_{B} \left( P_{\delta,\omega,\e} \partial_t \varphi + P_{\delta,\omega,\e} \vc{v}_{\delta,\omega, \e} \cdot \Grad_x \varphi + \vc{G_{P_{\delta,\omega, \e}}}  \varphi(t, \cdot) \right) dx dt}, \\ \\
 \hspace{0.5cm}\displaystyle{ \int_{B}  Q_{\delta,\omega, \e} \varphi(\tau,\cdot) \, dx - \int_{\Omega_0}  Q_0 \varphi(0,\cdot)dx} =\\
 \hspace{0.6cm}\displaystyle{\int_0^{\tau} \!\!\int_{B} \left( Q_{\delta,\omega, \e} \partial_t \varphi + P_{\delta,\omega, \e} \vc{v}_{\delta,\omega, \e} \cdot \Grad_x \varphi + \vc{G_{Q_{\delta,\omega, \e}}} \varphi(t, \cdot) \right) dx dt},\\ \\
 \hspace{0.5cm} \displaystyle{ \int_{B} D_{\delta,\omega, \e} \varphi(\tau,\cdot) \, dx   - \int_{\Omega_0}  D_0 \varphi(0,\cdot)dx}= \\
 \hspace{0.6cm}\displaystyle{\int_0^{\tau} \!\!\int_{B} \left( D_{\delta,\omega, \e} \partial_t \varphi + D_{\delta,\omega, \e} \vc{v}_{\delta,\omega, \e} \cdot \Grad_x \varphi + \vc{G_{D_{\delta,\omega, \e}}}  \varphi(t, \cdot) \right) dx dt }.
\end{array}
\right\}
%\tag {\bf{Ip}}
 \label{w-D}
\end{equation}
%In the sequel we refer to this set of relations as {\text{\bf (Ip)}}.

\item The weak formulation for the penalized Forchheimer's equation reads
\begin{equation} 
\begin{split}
 \int_{B} \vr\vc{v}\cdot\vc{\varphi}(\tau,\cdot)&\dx -\int_{B}(\vr\vc{v})_0 \cdot \vc{\varphi} (0, \cdot)dx\\
=\int_{0}^{\tau}\!\!\int_{B}\Big(\!\vr_{\delta,\omega,\e}\vc{v}_{\delta,\omega,\e}\cdot\partial_{t}\vc{\varphi} +\vr& \vc{v}_{\omega, \e, \delta} \otimes \vc{v}_{\omega, \e, \delta} : \Grad_x \vc{\varphi} + \sigma_{\delta,\omega, \e}\Div \vc{\varphi} \!\Big) dxdt\\
+\int_0^\tau \!\!\int_{B} \Big(\mu_{\omega}\Grad_x &\vc{v}_{\de,\omega, \e}:\Grad_x \vc{\varphi}  - \frac{{\mu}_{\omega}} {K} \vc{v}_{\de,\omega, \e} \vc{\varphi}\Big)dxdt\\
+ \frac{1}{\varepsilon} \int_{\Gamma_t} ((\vc{V} &- \vc{v}_{\delta,\omega, \e}) \cdot \vc{n} \vc{\varphi} \cdot \vc{n}) dS_x =0
\end{split}
 \label{w-pressure2}
\end{equation}
for any test function $\vc{\varphi} \in C_c^{\infty}(B; \mathbb{R}^3),$
where $\vc{v}_{\omega, \e} \in W_0^{1,2}(B; \mathbb{R}^3),$
and  $\vc{v}_{\omega, \e}$ satisfies the no-slip boundary condition 
\begin{equation}
\vc{v}_{\omega, \e}|_{\partial B} = 0\,\, \mbox{in the sense of traces}. \label{no-slip}
\end{equation} 
and $ \sigma_{\delta,\omega, \e}= \sigma_{\delta}(P_{\delta,\omega, \e}, Q_{\delta,\omega, \e}, D_{\delta,\omega, \e})$, 
\item The weak formulation for $C_{\de,\omega,\e}$ is as follows,
\begin{equation*}
\int_{B}  C_{\de,\omega, \e} \varphi(\tau,\cdot) \, dx - \int_{\Omega_{0}}  C_0 \varphi(0,\cdot)dx 
\end{equation*}
\begin{equation}
 = \int_0^{\tau}\!\! \int_{B}  \left( C_{\de,\omega, \e} \partial_t \varphi  -  \nu_{\omega} \Grad_x C_{\de,\omega, \e}\cdot \Grad_x \varphi  -  C_{\de,\omega, \e}  \varphi(\tau, \cdot) \right)dx dt,
\label{w-p-C}
\end{equation}
for any test function $\varphi \in C_c^{\infty}([0,T] \times \mathbb{R}^3)$ and $C_{\de,\omega, \e}$ satisfies the boundary conditions 
\begin{equation}
 C_{\de,\omega, \e}|_{\partial B} = 0\,\, \mbox{in the sense of traces}. \label{neumann}
\end{equation} 
\end{itemize}

Here, $\varepsilon$ and $\omega$ are positive parameters. 
%The choice of no-slip boundary condition \eqref{no-slip} is not essential.

\section{Uniform bounds}\label{S4}

The existence of global-in-time solutions $(P_{\delta, \omega, \e}, Q_{\delta, \omega, \e}, D_{\delta, \omega, \e}, \vc{v}_{\delta, \omega,\e}, C_{\delta, \omega, \e}) $ for  the penalized problem can be proved for fixed $\omega, \e, \delta$ with the method described in the Section 3.
In this section we collect all the uniform bounds satisfied by the solutions $(P_{\delta, \omega, \e}, Q_{\delta, \omega, \e}, D_{\delta, \omega, \e}, \vc{v}_{\delta, \omega, \e}, C_{\delta, \omega, \e}) $. We start by the nutrient equation. By applying standard theory for parabolic equations (see \cite{AS67}) we  obtain the following bounds for the nutrient 
$C_{\delta, \omega,\e}$ 
 \begin{equation}
\frac{\partial}{\partial t}\int_{B}\frac{1}{2}C^{2}_{\de,\omega, \e}dx +\int_{B}(C^{2}_{\de,\omega, \e}+\nu_{\omega}|\Grad_{x}C_{\de, \omega, \e}|^{2})dx=0. 
\label{mpC1}
\end{equation}
\begin{equation}
\|C_{\delta,\omega, \e}\|_{L^{\infty}([0,T]\times B)}\leq \max\{\|C_{0}\|_{L^{\infty}}, \bar{C}\}.
\label{mpC}
\end{equation}
Moreover, the constructed  solutions  $P_{\delta, \omega, \e}, Q_{\delta, \omega, \e}, D_{\delta, \omega, \e}$ satisfy the following {\em energy inequality}
\begin{equation}
\frac{\partial}{\partial t} \int_B \frac{1}{m-1} (P^m_{\delta, \omega,\e} + Q^m_{\delta, \omega,\e} + D^m_{\delta, \omega,\e})dx
\label{energy}
\end{equation}
\[+ \frac{\partial}{\partial t}\int_{B}\frac{\delta}{\beta-1} (P^\beta_{\delta, \omega,\e} + Q^\beta_{\delta,\omega,\e} + D^\beta_{\delta, \omega,\e}) dx +\frac{\partial}{\partial t} \int_B\frac{1}{2}\vr_{\delta, \omega, \e}|\vc{v}_{\delta, \omega, \e}|^{2}dx\]

%\[+\frac{\partial}{\partial t} \int_B\frac{1}{2}(P_{\delta, \omega, \e}+Q_{\delta, \omega, \e}+D_{\delta, \omega, \e})|\vc{v}_{\delta, \omega, \e}|^{2}dx\]
\[
+\int_B \frac{m}{m-1}\left[ (K_Q + K_{A})\bar{C} P^m_{\delta, \omega,\e} +  K_D \bar{C}Q^m_{\delta, \omega,\e} + K_D D^m_{\delta, \omega,\e}\right] dx \]
\[
+\int_B \frac{\delta\beta}{\beta-1} \left[ (K_Q + K_{A})\bar{C} P^\beta_{\delta, \omega,\e} +  K_D \bar{C}Q^\beta_{\delta, \omega,\e} + K_D D^\beta_{\delta, \omega,\e}\right] dx \]
\[+ \int_B \mu_{\omega} |\Grad_x \vc{v}_{\delta, \omega,\e}|^2  + \frac{\mu_{\omega}}{K}|\vc{v}_{\delta, \omega, \e}|^{2}dx
+
\frac{1}{\varepsilon} \int_{\Gamma_t} [(\vc{v}_{\delta, \omega,\e}-\vc{V}) \cdot \vc{n}] \vc{v}_{\delta,\omega,\e} \cdot \vc{n} dS \leq\]
\[
 \int_B  \frac{m}{m-1}\left[(K_B+K_Q+K_A) C_{\delta, \omega,\e}P^m_{\delta, \omega,\e}+(K_{D}-K_{P})C_{\delta, \omega,\e}Q^m_{\delta, \omega,\e}\right]dx\]
\[
 +\int_B \frac{\delta\beta}{\beta-1} \left[(K_B+K_Q+K_A) C_{\delta, \omega,\e}P^\beta_{\delta, \omega,\e}+(K_{D}-K_{P})C_{\delta, \omega,\e}Q^\beta_{\delta, \omega,\e}\right]dx\]
\[+\frac{1}{2}\int_{B}\left((K_Q+K_A) \bar{C}-(K_B+K_Q+K_A) C_{\delta, \omega,\e}\right)P_{\delta, \omega,\e}|\vc{v}_{\delta, \omega, \e}|^{2}dx\]
\[+\frac{1}{2}\int_{B}\left(K_D \bar{C}-(K_D-K_P) C_{\delta, \omega,\e}\right)Q_{\delta, \omega,\e}|\vc{v}_{\delta, \omega, \e}|^{2}dx+\frac{1}{2}\int_{B}K_RD_{\delta, \omega,\e}|\vc{v}_{\delta, \omega, \e}|^{2}dx.\]

Since the vector field $\vc{V}$ vanishes on the boundary of the reference domain $B$ it may be used as a text function in the weak formulation of the momentum equation  for the  penalized ForchheimerÕs equation  \eqref{w-pressure2}, namely

\begin{equation}
\int_B\vr_{\delta,\omega, \e} \vc{v}_{\delta,\omega,\e} \cdot \vc{V} (\tau, \cdot) dx - \int_B (\vr\vc{v})_0 \cdot \vc{V}(0, \cdot) dx
\label{weak=V}
\end{equation} 
\begin{equation*}
= \int_0^{\tau} \!\!\int_B \left(\vr_{\delta,\omega, \e} \ \vc{v}_{\omega,\e} \partial_t \vc{V}+\vr_{\delta,\omega, \e}\vc{v}_{\delta, \omega,\e} \otimes \vc{v}_{\delta, \omega,\e}: \Grad_x \vc{V}\right)dxdt
\end{equation*}
\[
 +\int_0^{\tau}\!\! \int_B \left(\sigma_{\delta,\omega,\e}\Div_x \vc{V}
- \mu_{\omega} \Grad_x \vc{v}_{\delta, \omega,\e}  : \Grad_x \vc{V} - \frac{\mu_{\omega}}{K} \vc{v}_{\delta,\omega,\e} \vc{V}   \right) dx dt\]
 \[
+\frac{1}{\varepsilon}\int_0^{\tau}\!\! \int_{\Gamma_t} ((\vc{V} - \vc{v}_{\delta,\omega,\e}) \cdot \vc{n} \vc{V} \cdot \vc{n} ) dS_x dt.
\]
Combining together \eqref{energy} with \eqref{weak=V} and by using \eqref{mpC} we get  the following {\em modified energy inequality},
\begin{equation}
\int_B \frac{1}{m-1} (P^m_{\delta, \omega,\e} + Q^m_{\delta, \omega,\e} + D^m_{\delta, \omega,\e})+ \frac{\delta}{\beta-1} (P^\beta_{\delta, \omega,\e} + Q^\beta_{\delta,\omega,\e} + D^\beta_{\delta, \omega,\e}) dx
\label{mod-energy}
\end{equation}
\[+\int_B\frac{1}{2}\vr_{\delta, \omega, \e}|\vc{v}_{\delta, \omega, \e}|^{2}dx\]
\[+\int_{0}^{\tau}\!\!\int_B \frac{m}{m-1}\left( K_{1} P^m_{\delta, \omega,\e} +  K_{2}Q^m_{\delta, \omega,\e} +K_{3} D^m_{\delta, \omega,\e}\right )dxdt \]
\[+\int_{0}^{\tau}\!\!\int_B \frac{\delta\beta}{\beta-1} \left(K_{4} P^\beta_{\delta, \omega,\e} +  K_{5}Q^\beta_{\delta, \omega,\e} +K_{6} D^\beta_{\delta, \omega,\e}\right)dxdt \]
\[+ \int_B \left(\mu_{\omega} |\Grad_x \vc{v}_{\delta, \omega,\e}|^2  + \frac{\mu_{\omega}}{K}|\vc{v}_{\delta, \omega, \e}|^{2}\right)dx
+\frac{1}{\varepsilon} \int_{\Gamma_t} |(\vc{v}_{\delta, \omega,\e}-\vc{V})\cdot\vc{n}|^{2}dSdt \leq\]
 \[\int_B \frac{1}{m-1} (P^m_{0} + Q^m_{0} + D^m_{0}) dx +\int_B \frac{\delta}{\beta-1} (P^\delta_{0} + Q^\delta_{0} + D^\delta_{0}) dx\]
\[+\int_B\frac{1}{2}\vr_{0}|\vc{v}_{0}|^{2}dx+\int_B\rho_{\delta,\omega, \e} \vc{v}_{\delta,\omega,\e} \cdot \vc{V} (\tau, \cdot) dx - \int_B (\vr\vc{v})_0 \cdot \vc{V}(0, \cdot) dx\] 
 \[+\int_0^{\tau} \!\!\int_B \left(\rho_{\delta,\omega, \e} \ \vc{v}_{\omega,\e} \partial_t \vc{V}+\rho_{\delta,\omega, \e}  [\vc{v}_{\delta, \omega,\e} \otimes \vc{v}_{\delta, \omega,\e}] : \Grad_x \vc{V} + \sigma_{\delta,\omega,\e}\Div_x \vc{V}\right) dxdt\]
 \[- \int_0^{\tau}\!\! \int_B\left (\mu_{\omega} \Grad_x \vc{v}_{\delta, \omega,\e}  : \Grad_x \vc{V} +\frac{\mu_{\omega}}{K} \vc{v}_{\delta,\omega,\e} \vc{V}   \right) dx dt \]
\[
 \int_{0}^{\tau}\!\!\int_B \Big( \frac{m}{m-1}\left(K_{7}P^m_{\delta, \omega,\e}+K_{8}Q^m_{\delta, \omega,\e}\right)
 + \frac{\delta\beta}{\beta-1} (K_9 P^\beta_{\delta, \omega,\e}+K_{10}Q^\beta_{\delta, \omega,\e})\Big)C_{\delta, \omega,\e}dxdt\]
\[+\frac{1}{2}\int_{0}^{\tau}\!\!\int_{B}\left((K_{11}P_{\delta, \omega,\e}+K_{12}Q_{\delta, \omega,\e})C_{\delta, \omega,\e}|\vc{v}_{\delta, \omega, \e}|^{2}
+K_{13}D_{\delta, \omega,\e}|\vc{v}_{\delta, \omega, \e}|^{2}\right)dxdt,\]
where $K_{i}$ are constants depending on $\bar{C}$, $K_{A}, K_{D}, K_{P}, K_{Q}, K_{R}$.

Since the vector field $\vc{V}$ is regular by applying the maximum principle \eqref{mpC} to $C_{\omega, \e}$ and  by means of Gronwall inequalities and \eqref{mpC1},  \eqref{mod-energy} we get the following uniform bounds  with respect to $\delta$, $\varepsilon$, $\omega$.
\begin{equation}
\begin{split}
&\|P_{\delta, \omega,\e}\|_{L^{\infty}(0,T;L^{m}(B))}+\|P_{\delta, \omega,\e}\|_{L^{m}(0,T;L^{m}(B))}\leq c,\\
&\delta\|P_{\delta, \omega,\e}\|_{L^{\infty}(0,T;L^{\beta}(B))}+\delta\|P_{\delta, \omega,\e}\|_{L^{\beta}(0,T;L^{\beta}(B))}\leq c,
\end{split}
\label{bP}
\end{equation}

\begin{equation}
\begin{split}
&\|Q_{\delta, \omega,\e}\|_{L^{\infty}(0,T;L^{m}(B))}+\|Q_{\delta,\omega,\e}\|_{L^{m}(0,T;L^{m}(B))}\leq c,\\
&\delta\|Q_{\delta, \omega,\e}\|_{L^{\infty}(0,T;L^{\beta}(B))}+\delta\|Q_{\delta, \omega,\e}\|_{L^{\beta}(0,T;L^{\beta}(B))}\leq c,
\end{split}
\label{bQ}
\end{equation}

\begin{equation}
\begin{split}
&\|D_{\delta, \omega,\e}\|_{L^{\infty}(0,T;L^{m}(B))}+\|D_{\omega,\e}\|_{L^{m}(0,T;L^{m}(B))}\leq c,\\
&\delta\|D_{\delta, \omega,\e}\|_{L^{\infty}(0,T;L^{\beta}(B))}+\delta\|D_{\delta, \omega,\e}\|_{L^{\beta}(0,T;L^{\beta}(B))}\leq c,
\end{split}
\label{bD}
\end{equation}

\begin{equation}
\begin{split}
\|\sqrt{P_{\delta, \omega,\e}}\vc{v}_{\delta,\omega,\e}\|_{L^{\infty}(0,T;L^{2}(B))}+\|&\sqrt{Q_{\delta, \omega,\e}}\vc{v}_{\delta,\omega,\e}\|_{L^{\infty}(0,T;L^{2}(B))}\\
&+\|\sqrt{D_{\delta,\omega,\e}}\vc{v}_{\delta,\omega,\e}\|_{L^{\infty}(0,T;L^{2}(B))}\leq c,
\end{split}
\label{bPQDv}
\end{equation}

\begin{equation}
\|\mu_{\omega} \vc{v}_{\delta, \omega,\e}\|_{L^{2}(0,T;L^{2}(B))}+ \|\mu_{\omega}\Grad\vc{v}_{\delta, \omega,\e}\|_{L^{2}(0,T;L^{2}(B))}\leq c
\label{bv}
\end{equation}

\begin{equation}
\|C_{\delta, \omega,\e}\|_{L^{2}(0,T;L^{2}(B))}+\|\nu_{\omega}\Grad C_{\delta, \omega,\e}\|_{L^{2}(0,T;L^{2}(B))}\leq c,
\label{bC}
\end{equation}

\begin{equation}
\int_{0}^{\tau}\!\!\int_{\Gamma_t} |(\vc{v}_{\delta, \omega,\e}-\vc{V}) \cdot \vc{n}|^{2}dSdt\leq c\varepsilon,
\label{bp}
\end{equation}

where $c$ depends only on the initial data.

\section{Pressure Estimates}\label{S5}
The a priori bounds should be at least so strong for all the expressions appearing in the weak formulation to make sense. As a matter of fact, slightly more is needed, namely the equi-integrability property in order to perform the limits with respect to the weak topology of the Lebesgue space $L^{1}$. It is evident that we can not control the pressure in the set $((0,T) \times B)\setminus \bar{Q}_t,$ where Forchheimer's  equation  contains a singular term. Nevertheless, local pressure estimates can be obtained in that region following the approach introduced by Lions \cite{Lions-1998} for the mathematical treatment of the Navier-Stokes equations for isentropic compressible fluids.  Here we only present a flavor of this  method which involves the use of test functions of the form 
$$ \varphi(t,x) =  \Grad_x \Delta_x^{-1}\left[{{\bf 1}}_B (P^{\nu}_{\de,\om, \e}+Q^{\nu}_{\de,\om, \e}
+D^{\nu}_{\de,\om, \e})\right],$$
in the weak formulation of the momentum equation. Here  $\nu>0$ is a small positive number, 
and the symbol $\Delta_x$ denotes the Laplace operator considered on the whole domain $\R^3$.

It needs to be emphasized that the estimates presented below are obtained on a  compact set ${\mathcal K}$ specially designed such that the property \eqref{K} below holds true. This property is crucial in dealing with the moving domain since it guarantees that the compact set ${\mathcal K}$ has no intersection with the boundary $\Gamma_{\tau}$ for all times  $\tau.$ Without this delicate choice of ${\mathcal K}$ the treatment of the singular term in  \eqref{w-pressure2} would be problematic having only the estimate \eqref{bp} in our disposal.

Since $m > \frac{3}{2}$ the estimates obtained earlier  will assist us in obtaining the bound$$
\int_0^T\!\! \int_{\mathcal K}   \left( P^{m+\nu}_{\de, \omega,\e} + Q^{m+\nu}_{\de, \omega, \e} + D^{m+\nu}_{\de, \omega,\e} \right)  + \delta \left(P^{\beta+\nu}_{\de, \omega,\e} + Q^{\beta+\nu}_{\de,\omega,\e} + D^{\beta +\nu }_{\de, \omega, \e} \right) ~dxdt \le c(T) $$
for any compact $\mathcal{K} \subset [0,T]\times\overline{B}$, such that
 \begin{equation}\label{K}
 \displaystyle{\mathcal{K}\cap\Big(\bigcup_{\tau\in[0,T]}(\{\tau\}\times \Gamma_{\tau})\Big)=\emptyset}.
 \end{equation}
Indeed writing  
\begin{equation*}
    \begin{split}
  &  \int_0^T\!\!\int_{\mathcal{K}} (P_{\de,\omega,\e}^{m+\nu}+ Q_{\de,\omega,\e}^{m+\nu}+ D_{\de,\omega,\e}^{m+\nu})~dxdt=\\
   &  - \int_{\mathcal{K}} (\vr\vc{v})_0 \vc{\varphi}~dx-\int_0^T\!\!\int_{\mathcal{K}}  (\vr_{\de,\omega,\e} \vc{v}_{\de,\e,\omega}) \partial_t \vc{\varphi}+ \vr_{\de,\omega,\e} \vc{v}_{\de,\omega,\e} \otimes \vc{v}_{\de,\omega,\e} :\Grad \vc{\varphi}~dxdt \\
  & + \int_0^T\!\!\int_{\mathcal{K}} \mu \Grad \vc{v}_{\de,\omega,\e} \Grad \vc{\varphi}~dxdt    + \frac{\mu}{K} \int_0^T\!\! \int_{\mathcal{K}}  \Div \vc{v}_{\de,\omega,\e}\vc{\varphi}~dxdt= \\
 &- \int_0^T\!\! \int_{\mathcal{K}} P_{\de,\omega,\e}^m (Q_{\de,\omega,\e}^{\nu}+ D_{\de,\omega,\e}^{\nu}) +\delta P_{\de,\omega,\e}^\beta (Q_{\de,\omega,\e}^{\nu}+ D_{\de,\omega,\e}^{\nu}) ~dxdt\\
 &- \int_0^T\!\! \int_{\mathcal{K}}  Q_{\de,\omega,\e}^{m} (P_{\de,\omega,\e}^{\nu}+ D_{\de,\omega,\e}^{\nu}) +\delta Q_{\de,\omega,\e}^{\beta} (P_{\de,\omega,\e}^{\nu}+ D_{\de,\omega,\e}^{\nu})~dxdt\\
 &- \int_0^T\!\! \int_{\mathcal{K}} D_{\de,\omega,\e}^{m} (P_{\de,\omega,\e}^{\nu}+ Q_{\de,\omega,\e}^{\nu}) +\delta D_{\de,\omega,\e}^{\beta} (P_{\de,\omega,\e}^{\nu}+ Q_{\de,\omega,\e}^{\nu})    ~dxdt.
    \end{split}
\end{equation*}
Taking into consideration 
$$\partial_t (P^{\nu}_{\de,\omega,\e} + Q^{\nu}_{\de,\omega,\e} + D^{\nu}_{\de,\omega,\e})+ \Div ((P^{\nu}_{\de,\omega,\e} + Q^{\nu}_{\de,\omega,\e} + D^{\nu}_{\de,\omega,\e})  \vc{v}_{\de,\omega,\e}) $$
$$+ (\nu -1)(P^{\nu}_{\de,\omega,\e} + Q^{\nu}_{\de,\omega,\e} + D^{\nu}_{\de,\omega,\e})\Div \vc{v} = \vc{G}$$
 where  $\vc{G}$ is a function of $P^{\nu}_{\de,\omega,\e}, Q^{\nu}_{\de,\omega,\e}, D^{\nu}_{\de,\omega,\e}, C_{\de,\omega,\e}$  and thanks to the uniform bounds of the previous section $\vc{G} \in L^{\infty}([0,T]; L^{p}(\mathcal{K}))$, for  $p>1$, so we get that 
 $$ \partial_t \varphi\quad \text{is bounded in $L^p(0,T;L^q(K))$ for appropriate $1\le p, q  \le \infty$}.$$
The remaining terms can be controlled  by following standard arguments (see \cite{CKT} Section 3.5.2).
The pressure estimates can be extended ``up to the boundary'' provided we are able to construct
suitable test functions in the momentum equation. More precisely, we need  
$\varphi = \varphi(t,x)$ such that 
\begin{itemize}
\item $\partial_t \varphi, \Grad_x \varphi$ belong to $L^q(Q_t)$ for a given (large) $q \ll 1;$ 
\item $\varphi (t, \cdot) \in W_0^{1,q}(\Omega_t; \R^3)$ for any $\tau \in (0,T);$
\item $\varphi (T;\cdot)  = 0;$
\item  $\Div_x \varphi(t,x) \to \infty$  for  $x \to \partial \Omega_t$ uniformly for  $t$  in  compact  subsets of $(0, T ).$
\end{itemize}
For the construction of such $\varphi$ we refer the reader to Feireisl \cite{FeireislNS-2011}.

\section{Vanishing penalization $\varepsilon \to 0$}\label{S6}
In this section we start performing  the limits of our three level approximation. 
The first step is to keep $\delta$ and $\omega$ fixed and to perform the penalization limit $\e\to 0$.
The main issues of this process will be to recover the strong convergence of $P_{\delta,\omega,\e}$,  $Q_{\delta,\omega,\e}$, $D_{\delta,\omega,\e}$  and to get rid of the quantities that are supported by the  healthy tissue $B\backslash \Omega_{t}$.

As a consequence of the uniform bounds \eqref{bP}-\eqref{bp} we get that the weak solutions of our approximating system satisfy
\begin{equation}
\left.
\begin{array}{r}
P_{\delta, \omega,\e}\longrightarrow P_{\delta,\omega}\\ \\
Q_{\delta,\omega,\e}\longrightarrow Q_{\delta,\omega}\\ \\
D_{\delta,\omega,\e}\longrightarrow D_{\delta,\omega}\\ \\
C_{\delta,\omega,\e}\longrightarrow C_{\delta,\omega}
\end{array}
\right\}\quad \text{in} \quad C_{\text{weak}}(0,T;L^{m}(B))
\label{cPQDC}
\end{equation}
From the bounds \eqref{bv} and \eqref{bC}  we get
\begin{equation}
\vc{v}_{\delta,\omega,\e}\rightarrow \vc{v}_{\delta,\omega} \quad  \text{weakly in $L^{2}(0,T;W^{1,2}_{0}(B))$} 
\label{cv}
\end{equation}
\begin{equation}
C_{\delta,\omega,\e}\rightarrow C_{\delta,\omega}
\quad  \text{weakly in $L^{2}(0,T;W^{1,2}_{0}(B))$} 
\label{cC1}
\end{equation}
while from \eqref{bp} we have that
\begin{equation}
(\vc{v}_{\delta,\omega, \e}-\vc{V}) \cdot \vc{n}(\tau, \cdot)\big|_{\Gamma_{\tau}}=0\quad \text{for a.a $\tau\in [0,T]$.}\nonumber
\label{cp}
\end{equation}
By combining together \eqref{bP}, \eqref{bQ}, \eqref{bD}, \eqref{bPQDv}, \eqref{bv} and the compact embedding of $L^{m}(B)$ in $W^{-1,2}(B)$we get 
\begin{equation}
\hspace{-0.55cm}
\left.
\begin{array}{r}
P_{\delta,\omega,\e}\vc{v}_{\delta,\omega,\e}\rightarrow P_{\delta,\omega}\vc{v}_{\delta,\omega}\\ \\
Q_{\delta,\omega,\e}\vc{v}_{\delta,\omega,\e}\rightarrow Q_{\delta,\omega}\vc{v}_{\delta,\omega}\\ \\
D_{\delta,\omega,\e}\vc{v}_{\delta,\omega,\e}\rightarrow D_{\delta,\omega}\vc{v}_{\delta,\omega}
\end{array}
\right\}\  \text{weakly-($\ast$) in} \ L^{\infty}(0,T;L^{2m/m+2}(B)).
\label{cPQDCv}
\end{equation}
Finally from the equations \eqref{dP}-\eqref{dD} it follows that
\begin{equation}
\hspace{-0.55cm}
\left.
\begin{array}{r}
P_{\delta,\omega,\e}\vc{v}_{\delta,\omega,\e}\rightarrow P_{\delta,\omega}\vc{v}_{\delta,\omega}\\ \\
Q_{\delta,\omega,\e}\vc{v}_{\delta,\omega,\e}\rightarrow Q_{\delta,\omega}\vc{v}_{\delta,\omega}\\ \\
D_{\delta,\omega,\e}\vc{v}_{\delta,\omega,\e}\rightarrow D_{\delta,\omega}\vc{v}_{\delta,\omega}
\end{array}
\right\}\quad\text{in} \ C_{\text{weak}}([T_{1},T_{2}];L^{2m/m+2}(B)),
\label{cPQDCv2}
\end{equation}
Since the embedding of $W^{1,2}_{0}(B)$ in $L^{6}(B)$ is compact we have that
\begin{equation*}
\hspace{-0.55cm}
\left.
\begin{array}{r}
P_{\delta,\omega,\e}\vc{v}_{\omega,\e}\otimes \vc{v}_{\omega,\e}\rightarrow \overline{P_{\delta,\omega}\vc{v}_{\delta,\omega}\otimes\vc{v}_{\delta\omega}}\\ \\
Q_{\delta,\omega,\e}\vc{v}_{\omega,\e}\otimes \vc{v}_{\omega,\e}\rightarrow \overline{Q_{\delta,\omega}\vc{v}_{\delta,\omega}\otimes\vc{v}_{\delta\omega}}\\ \\
D_{\delta,\omega,\e}\vc{v}_{\omega,\e}\otimes \vc{v}_{\omega,\e}\rightarrow \overline{D_{\delta,\omega}\vc{v}_{\delta,\omega}\otimes\vc{v}_{\delta\omega}}
\end{array}
\right\}\quad \text{weakly in $L^{2}(0,T;L^{6m/4m+3}(B))$,}
\end{equation*}
where the bar denotes the weak limit of the nonlinear functions. As in \eqref{cPQDCv2} we can conclude that 
\begin{equation}
\hspace{-0.55cm}
\left.
\begin{array}{r}
 \overline{P_{\delta,\omega}\vc{v}_{\delta,\omega}\otimes\vc{v}_{\delta\omega}}=P_{\delta,\omega}\vc{v}_{\delta,\omega}\otimes\vc{v}_{\delta\omega}\\ \\
 \overline{Q_{\delta,\omega}\vc{v}_{\delta,\omega}\otimes\vc{v}_{\delta\omega}}=Q_{\delta,\omega}\vc{v}_{\delta,\omega}\otimes\vc{v}_{\delta\omega}\\ \\
\overline{D_{\delta,\omega}\vc{v}_{\delta,\omega}\otimes\vc{v}_{\delta\omega}}=D_{\delta,\omega}\vc{v}_{\delta,\omega}\otimes\vc{v}_{\delta\omega}
\end{array}
\right\}\quad \text{a.e. in  $(0,T)\times B$.}
\label{conv}
\end{equation}

Taking into account \eqref{bP}-\eqref{bD}, \eqref{bC} and, as before,  the compact embedding of $L^{m}(B)$ in $W^{-1,2}(B)$ we get
\begin{equation}
\hspace{-0.55cm}
\left.
\begin{array}{r}
P_{\delta, \omega,\e}C_{\delta,\omega,\e}\longrightarrow P_{\delta,\omega}C_{\delta,\omega}\\ \\
Q_{\delta,\omega,\e}C_{\delta,\omega,\e}\longrightarrow Q_{\delta,\omega}C_{\delta,\omega}\\ \\
D_{\delta,\omega,\e}C_{\delta,\omega,\e}\longrightarrow D_{\delta,\omega}C_{\delta,\omega}
\end{array}
\right\}\ \text{weakly-($\ast$) in $L^{\infty}(0,T;L^{2m/m+2}(B))$.}
\label{cC}
\end{equation}

By using  \eqref{cPQDC}, \eqref{cv}, \eqref{cC1}, \eqref{cPQDCv}, \eqref{cPQDCv2} and \eqref{cC} we can pass to the limit in the weak formulations \eqref{w-D} and \eqref{w-p-C} and we obtain 
\begin{equation}
\left.
\begin{array}{l}
\displaystyle{\int_{B}   P_{\delta,\omega} \varphi(\tau,\cdot)dx  - \int_{B}  P_0 \varphi(0,\cdot)  dx =}\\
\hspace{1.5cm}\displaystyle{ \int_0^{\tau}\! \!\int_{B} \left( P_{\delta,\omega} \partial_t \varphi + P_{\de,\omega} \vc{v} \cdot \Grad_x \varphi + \vc{G_{P_{\delta,\omega}}}  \varphi(t, \cdot) \right) dx dt,} \\ \\
\displaystyle{\int_{B}  Q_{\delta,\omega} \varphi(\tau,\cdot)dx -  \int_B Q_0 \varphi(0,\cdot) dx =}\\ 
\hspace{1.5cm}\displaystyle{  \int_0^{\tau}\! \!\int_{B} \left( Q_{\delta,\omega} \partial_t \varphi +Q_{\delta,\omega}\vc{v} \cdot \Grad_x \varphi + \vc{G_{Q_{\delta,\omega}}}  \varphi(t, \cdot) \right) dx dt,}\\ \\
\displaystyle{\int_{B} D_{\delta,\omega} \varphi(\tau,\cdot) dx -  \int_B D_0 \varphi(0,\cdot) dx=} \\
\hspace{1.5cm}\displaystyle{\int_0^{\tau}\!\! \int_{B} \left( D_{\delta,\omega} \partial_t \varphi + D_{\delta,\omega} \vc{v} \cdot \Grad_x \varphi + \vc{G_{D_{\delta,\omega}}}  \varphi(t, \cdot) \right) dx dt, }\\ \\
\displaystyle{ \int_{B} C_{\delta,\omega} \varphi(\tau,\cdot) dx - \int_B  C_0 \varphi(0,\cdot) dx=}\\
\vspace{0.1in}
\hspace{1.5cm}\displaystyle{\int_0^{\tau}\!\! \int_{B} \left(C_{\delta,\omega} \partial_t \varphi - \nu_{\omega}\Grad_{x} C_{\delta,\omega} \cdot \Grad_x \varphi - C_{\delta,\omega} \varphi(t, \cdot) \right) dx dt. }
\end{array}
\right\}
%\tag {\bf{IIp}}
 \label{p-D}
\end{equation}

Passing into the limit in the weak formulation \eqref{w-pressure2} of the Forchenheimer's  equation  we get

\begin{equation} \nonumber
\int_{B}\vr_{\delta,\omega}\vc{v}_{\delta,\omega}\vc{\varphi}(\tau,\cdot)dx- \int_{B}{ \vc{m}_0 \cdot \vc{\varphi} (0, \cdot)\,\dx}
\end{equation}
\begin{equation*}
= \int_0^{\tau}\!\! \int_{B} { \Big( \vr_{\delta,\omega} \vc{v}_{\delta,\omega} \cdot
\partial_t \vc{\varphi} + \vr \vc{v}_{\delta, \omega,} \otimes \vc{v}_{ \delta, \omega} : \Grad_x \vc{\varphi} + \overline{\sigma_{\delta} (P_{\delta, \omega},Q_{\delta, \omega},D_{\delta, \omega})} \ \Div \vc{\varphi} \Big)}dxdt 
\end{equation*}
\begin{equation}
 -\int_0^\infty \int_{B} \left(\mu_{\omega}
 \Grad_x \vc{v}_{\delta, \omega}:\Grad_x \vc{\varphi} + \frac{{\mu}_{\omega}} {K} \vc{v}_{\delta,\omega} \vc{\varphi}\right)\,  \dxdt 
\label{w-pressure22}
\end{equation}
for any test function $\vc{\varphi} \in C_c^{\infty}(B; \mathbb{R}^3).$

\subsection{Strong convergence of the densities}
As we can see in \eqref{w-pressure22} the convergence properties obtained so far are not enough in order to pass into the limit in the pressure term.  Therefore, we need to establish the strong convergence of the density of the proliferating, quiescent and dead cells. The main steps of our approach are summarized  below.

\begin{itemize}

\item First we establish that the {\em effective viscous pressures} 
\[P^m_{\de,\om,\e}+\delta P^\beta_{\de,\om,\e} -2 \mu_{\om} \Div_x \vc{v}_{\de,\om,\e}\]
\[Q^m_{\de,\om,\e}+\delta Q^\beta_{\de,\om,\e} -2 \mu_{\om} \Div_x \vc{v}_{\de,\om,\e}\]
\[D^m_{\de,\om,\e}+\delta D^\beta_{\de,\om,\e} -2 \mu_{\om} \Div_x \vc{v}_{\de,\om,\e}\]
%$$P^m_{\de,\om,\e}+Q^m_{\de,\om,\e}+D^m_{\de,\om,\e} +\delta (P^\beta_{\de,\om,\e}+Q^\beta_{\de,\om,\e}+D^\beta_{\de,\om,\e}) -2 \mu_{\om} \Div_x \vc{v}_{\de,\om,\e}$$
are weakly continuous.
\item Next we obtain a control on the amplitude of oscillations or of the following  {\em oscillations defect measure}, showing that 
$$
\sup_{k\geq 0}\left(\limsup_{\e \to 0}\|T_k(Z_{\de,\om,\e}) - T_k(Z_{\de,\om}) \|_{L^{m +1}((0,T) \times B)} \right)\le c,
$$
with $T_k(\cdot)$ defined in \eqref{cut} and $Z_{\de,\om,\e}$ stands for $P_{\de,\om,\e}$, $Q_{\de,\om,\e}$,  $D_{\de,\om,\e}$.
\item Finally we show the decay of the {\em  defect measure}:
%\begin{equation}
%\begin{split}
$$\int_{B}(\overline{Z_{\de,\om,\e}\log Z_{\de,\om}}-Z_{\de,\om}\log Z_{\de,\om})(t)dx.$$
%\\
%\leq\int_{0}^{t}\!\!\int_{B}(T_{k}(P_{\de,\om})&-\overline{T_{k}(P_{\de,\om})})\Div\vc{v}_{\de,\om}dxdt. \label{odm}
%\end{split}
%\end{equation}

\end{itemize}

\subsubsection{Preliminary material}
 Now we  establish some preliminary properties of the equations satisfied by $P_{\delta, \omega}$,  $Q_{\delta, \omega}$,  $D_{\delta, \omega}$ that will be useful in the sequel. First we define a family of cut-off functions,
\begin{equation}\label{cut}
T_{k}(z)=kT\left(\frac{z}{k}\right), \quad \text {for $z\in \R$, $k=1,2,\ldots$,}
\end{equation}
where $T\in C^{\infty}(\R)$, 
$$T(z)=z \ \text{for $|z|\leq 1$,}\ T(z)=2\ \text{for $z\geq 3$,}$$
$T$ is concave on $[0,\infty)$ and $T(-z)=-T(z)$.
In order to simplify the notations we will rewrite the terms  ${\bf \{G_P, G_Q\}}$ in \eqref{G} as follows
\begin{equation}
\begin{split} \label{GG}
& \vc{G_P} =  F(C) P \\
& \vc{G_Q} =   F(C) Q\\
\end{split}
\end{equation}
where $F$ denotes a linear function of $C$.
First, we consider the balance equation satisfied by $P_{\delta,\omega,\e}$, it is straightforward to prove that the following relation holds in $\mathcal{D}'((0,T);\R^{3})$
\begin{equation}
\begin{split}
\partial_{t} T_{k}(P_{\delta,\omega,\e}) &+\Div(T_{k}(P_{\delta,\omega,\e}) \vc{v}_{\delta,\omega,\e})+(T'_{k}(P_{\delta,\omega,\e})P_{\delta,\omega,\e}-T_{k}(P_{\delta,\omega,\e}))\Div\vc{v}_{\delta,\omega,\e}\\&=T'_{k}(P_{\delta,\omega,\e})P_{\delta,\omega,\e}F(C_{\delta,\omega,\e})
\end{split}
\label{TPe}
\end{equation}
If we take into account \eqref{cPQDC}-\eqref{cC} and take the limit $\e\to 0$ we have
\begin{equation}
\begin{split}
\partial_{t} \overline{T_{k}(P_{\delta,\omega})} &+\Div(\overline{T_{k}(P_{\delta,\omega})}\vc{v}_{\delta,\omega})+\overline{(T'_{k}(P_{\delta,\omega})P_{\delta,\omega}-T_{k}(P_{\delta,\omega}))\Div\vc{v}_{\delta,\omega}}\\&=\overline{T'_{k}(P_{\delta,\omega})P_{\delta,\omega}}F(C_{\delta,\omega}),
\end{split}
\label{TP}
\end{equation}
where 
$$
\left.
\begin{array}{l}
T_{k}(P_{\delta,\omega,\e})\rightarrow\overline{T_{k}(P_{\delta,\omega})}\\
T'_{k}(P_{\delta,\omega,\e})P_{\delta,\omega,\e}\rightarrow \overline{T'_{k}(P_{\delta,\omega})P_{\delta,\omega}}
\end{array}
\right\} \ \text{in $C(0,T;L^{p}_{weak}(B))$, for all $1\leq p<\infty$}$$
and 
$$T'_{k}(P_{\delta,\omega,\e})P_{\delta,\omega,\e}-T_{k}(P_{\delta,\omega,\e}))\Div\vc{v}_{\delta,\omega,\e}\rightarrow\overline{(T'_{k}(P_{\delta,\omega})P_{\delta,\omega}-T_{k}(P_{\delta,\omega}))\Div\vc{v}_{\delta,\omega}},$$
weakly in $L^{2}((0,T)\times B)$.

In a similar way we have the following relations hold in $\mathcal{D}'((0,T);\R^{3})$
\begin{equation}
\begin{split}
\partial_{t} T_{k}(Q_{\delta,\omega,\e}) &+\Div(T_{k}(Q_{\delta,\omega,\e}) \vc{v}_{\delta,\omega,\e})+(T'_{k}(Q_{\delta,\omega,\e})Q_{\delta,\omega,\e}\!\!-\!\!T_{k}(Q_{\delta,\omega,\e}))\Div\vc{v}_{\delta,\omega,\e}\\&=T'_{k}(q_{\delta,\omega,\e})Q_{\delta,\omega,\e}F(C_{\delta,\omega,\e}),
\end{split}
\label{TQe}
\end{equation}

\begin{equation}
\begin{split}
\partial_{t} \overline{T_{k}(Q_{\delta,\omega})} &+\Div(\overline{T_{k}(Q_{\delta,\omega})}\vc{v}_{\delta,\omega})+\overline{(T'_{k}(Q_{\delta,\omega})Q_{\delta,\omega}-T_{k}(P_{\delta,\omega}))\Div\vc{v}_{\delta,\omega}}\\&=\overline{T'_{k}(Q_{\delta,\omega})Q_{\delta,\omega}}F(C_{\delta,\omega}),
\end{split}
\label{TQ}
\end{equation}

\begin{equation}
\begin{split}
\partial_{t} T_{k}(D_{\delta,\omega,\e}) &+\Div(T_{k}(D_{\delta,\omega,\e}) \vc{v}_{\delta,\omega,\e})+(T'_{k}(D_{\delta,\omega,\e})D_{\delta,\omega,\e}-T_{k}(D_{\delta,\omega,\e}))\Div\vc{v}_{\delta,\omega,\e}\\&=-K_{R}T'_{k}(D_{\delta,\omega,\e})D_{\delta,\omega,\e},
\end{split}
\label{TDe}
\end{equation}

\begin{equation}
\begin{split}
\partial_{t} \overline{T_{k}(D_{\delta,\omega})} &+\Div(\overline{T_{k}(D_{\delta,\omega})}\vc{v}_{\delta,\omega})+\overline{(T'_{k}(Q_{\delta,\omega})Q_{\delta,\omega}-T_{k}(P_{\delta,\omega}))\Div\vc{v}_{\delta,\omega}}\\&=-K_{R}\overline{T'_{k}(D_{\delta,\omega})D_{\delta,\omega}}.
\end{split}
\label{TD}
\end{equation}

\subsubsection{The weak continuity of the effective viscous pressure}
In this section we show that the quantities
\[P^m_{\de,\om,\e} +\delta P^\beta_{\de,\om,\e}  -2 \mu_{\om} \Div_x \vc{v}_{\de,\om,\e} \]
\[Q^m_{\de,\om,\e} +\delta Q^\beta_{\de,\om,\e} -2 \mu_{\om} \Div_x \vc{v}_{\de,\om,\e} \]
\[D^m_{\de,\om,\e} +\delta D^\beta_{\de,\om,\e} -2 \mu_{\om} \Div_x \vc{v}_{\de,\om,\e} \]
known as ``effective viscous pressure" exhibit  certain ``weak continuity" as it is established by the following proposition.
\begin{proposition}
\label{P1.1}
Under the hypothesis of Theorem \ref{T2.2}, we have
\begin{equation}
\begin{split}
\lim_{\e\to 0} \int_0^T \!\!\int_{\mathcal{K}} \psi \phi \left(P_{\de,\om,\e}^m +\delta P_{\de,\om,\e}^\delta - 2\mu_{\om} \Div_x \vc{v}_{\de,\om,\e}\right) T_k(P_{\de,\om,\e})dxdt\\
=\int_0^T \!\!\int_{\mathcal{K}} \psi \phi \left(\overline{P_{\de,\om}^m +\delta P_{\de,\om}^\delta}-2\mu_{\om} \Div_x \vc{v}_{\de,\om}\right) \overline{T_k(P_{\de,\om})}dxdt.
\end{split}
\label{evpP}
\end{equation}
\begin{equation}
\begin{split}
\lim_{\e\to 0} \int_0^T \!\!\int_{\mathcal{K}} \psi \phi \left(Q_{\de,\om,\e}^m +\delta Q_{\de,\om,\e}^\delta - 2\mu_{\om} \Div_x \vc{v}_{\de,\om,\e}\right) T_k(Q_{\de,\om,\e})dxdt\\
=\int_0^T \!\!\int_{\mathcal{K}} \psi \phi \left(\overline{Q_{\de,\om}^m +\delta Q_{\de,\om}^\delta}-2\mu_{\om} \Div_x \vc{v}_{\de,\om}\right) \overline{T_k(Q_{\de,\om})}dxdt.
\end{split}
\label{evpQ}
\end{equation}
\begin{equation}
\begin{split}
\lim_{\e\to 0} \int_0^T \!\!\int_{\mathcal{K}} \psi \phi \left(D_{\de,\om,\e}^m +\delta D_{\de,\om,\e}^\delta - 2\mu_{\om} \Div_x \vc{v}_{\de,\om,\e}\right) T_k(D_{\de,\om,\e})dxdt\\
=\int_0^T \!\!\int_{\mathcal{K}} \psi \phi \left(\overline{D_{\de,\om}^m +\delta D_{\de,\om}^\delta}-2\mu_{\om} \Div_x \vc{v}_{\de,\om}\right) \overline{T_k(D_{\de,\om})}dxdt.
\end{split}
\label{evpD}
\end{equation}
for any $\psi \in {\mathcal D}(0,T), \phi \in {\mathcal D}(\mathcal{K})$ and for any compact $\mathcal{K} \subset [0,T]\times\overline{B}$, such that $\displaystyle{\mathcal{K}\cap\Big(\bigcup_{\tau\in[0,T]}(\{\tau\}\times \Gamma_{\tau})\Big)=\emptyset}$.
\end{proposition}

\begin{proof}
Consider the operators
$${\mathcal A}_j[\vc{v}] = \Delta^{-1}\partial_{x_j}[\vc{v}], \,\, j=1,2,3$$
specifically
$${\mathcal A}_j[\vc{v}] = {\mathcal F}^{-1} \left\{ \frac{-i \xi_j}{|\xi|^2} {\mathcal F}[v](\xi) \right\}, \,\, j=1,2,3,$$
where ${\mathcal F}$ denotes the Fourier transform. These operators are endowed with some nice properties, namely
$$
\begin{cases}
& \|\partial_{x_i} {\mathcal A}_j[\vc{v}]\|_{W^{1,s}(B)} \le c(s,B) \|\vc{v}\|_{L^s(\R^3)}, \,\, 1 <s<\infty\\
& \|{\mathcal A}_i[\vc{v}]\|_{L^q(B)} \le c(q,s,B) \|{\vc{v}}\|_{L^s(\R^3)}, \,\, q \,\,  \mbox{finite},\frac{1}{q}\ge \frac{1}{s} -\frac{1}{3},\\
&\|{\mathcal A}_i[\vc{v}]\|_{L^{\infty}(B)} \le c(s, B) \|\vc{v}\|_{L^s(\R^3)} \,\, \mbox{if} \,\,  s>3.
\end{cases}
$$
Now, we use the quantities
$$\varphi_i(t,x) = \psi(t) \phi(x) {\mathcal A}_i[T_k(P_{\de,\om,\e})], \,\, \psi\in {\mathcal D}(0,T),\,\, \phi \in {\mathcal D}(\mathcal{K}), \,\, i=1,2,3.$$
 as test functions in the weak formulation \eqref{w-pressure2} of the penalized Forchheimer equation. After some analysis and by using the relation \eqref{TPe} we get
\begin{equation}
\int_0^T \!\!\int_{\mathcal{K}} \psi \phi 
 \left(P_{\de,\om,\e}^m + \de P_{\de,\om,\e}^\beta  -  2\mu_{\omega} \Div_x \vc{v}_{\de,\om,\e}\right) T_k(P_{\de,\om,\e})~dxdt =
 \label{wc1}
 \end{equation}
 \[\int_0^T\!\! \int_{\mathcal{K}} \psi \phi 
 \left(Q_{\de,\om,\e}^m +D_{\de,\om,\e}^m+ \de(Q_{\de,\om,\e}^\beta +D_{\de,\om,\e}^\beta) \right) T_k(P_{\de,\om,\e})~dxdt\]
 \[- \int_0^T\!\! \int_{\mathcal{K}} \psi \left[ P_{\de,\om,\e}^m +Q_{\de,\om}^m +D_{\de,\om,\e}^m + \de(P_{\de,\om,\e}^\beta +Q_{\de,\om,\e}^\beta +D_{\de,\om,\e}^\beta) \right] \partial_{x_i} \phi {\mathcal A}_i[T_k(P_{\de, \om,\e})]~dxdt\]
 \[+ \mu_{\omega} \int_0^T \!\!\int_{\mathcal{K}}\psi  \partial_{x_j}\phi \partial_{x_j}\vc{v}^i_{\de,\om,\e} {\mathcal A}_i[T_k(P_{\de,\om,\e})] ~dxdt\]
 \[- \mu_{\om} \int_0^T\!\! \int_{\mathcal{K}}\psi \left\{\vc{v}_{\de,\om,\e}^i\partial_{x_j}\phi \partial_{x_j} {\mathcal A}_i[T_k(P_{\de,\om,\e})] + \vc{v}_{\de,\om,\e}^i \partial_{x_i} \phi T_k(P_{\de,\om,\e})\right\}~dxdt\]
 \[
 -\frac{\mu_{\om}}{K} \int_0^T\!\! \int_{\mathcal{K}}   \vc{v}_{\de,\om,\e} \psi(t) \phi(x) {\mathcal A}[T_k(P_{\de,\om})] ~dxdt -\int_0^T\!\! \int_{\mathcal{K}}\phi  \vr_{\de,\om,\e}  \vc{v}_{\de,\om,\e}^i \partial_t \psi {\mathcal A}_i[T_k(P_{\de,\om,\e})] ~ dxdt
 \]
 \[ - \int_0^T\!\! \int_{\mathcal{K}}\phi  \vr_{\de,\om}  \vc{v}_{\de,\om}^i  \psi {\mathcal A}_i [(T_k(P_{\de,\om,\e}) - T'_k(P_{\de,\om,\e})P_{\de,\om,\e}) \Div_x \vc{v}_{\de,\om,\e}] ~dxdt\]
\[ + \int_0^T \!\!\int_{\mathcal{K}}\phi  \vr_{\de,\om}  \vc{v}_{\de,\om}^i \psi  T'_k(P_{\de,\om,\e})P_{\de,\om,\e}F(C_{\de,\om,\e})  ~dxdt\]
\[
 -\int_0^T \!\!\int_{\mathcal{K}}\psi \vr_{\de,\om,\e} \vc{v}_{\de,\om,\e}^i    \vc{v}_{\de,\om,\e}^j \partial_{x_j} \phi {\mathcal A}_i[T_k(P_{\de,\om,\e})]~dxdt
 \]
\[
+ \int_0^T \!\!\int_{\mathcal{K}}\psi  \vc{v}_{\de,\om,\e}^i   \left\{T_k(P_{\de,\om,\e}) {\mathcal R}_{i,j}[\phi \vr_{\de,\om,\e}\vc{v}_{\de,\om,\e}^j ] - \phi \vr_{\de,\om,\e} \vc{v}^j_{\de, \om,\e} {\mathcal R}_{i,j}[T_k(P_{\de,\om,\e})]\right\}~dxdt
 \]
where the operators ${\mathcal R}_{i,j}$ are defined as  $\partial_{x_{j}}{\mathcal A}_{i}[\vc{v}] $.

Analogously, we can repeat the above argument considering the equation \eqref{TP} and the following one,
\begin{equation}
\begin{split}
\partial_t (\vr \vc{v}_{\de,\om}) &+ \Div (\vr \vc{v}_{\de,\om} \otimes \vc{v}_{\de,\om}) + \overline{\Grad(P_{\de,\om}^{m}+\delta P_{\de,\om}^{\beta})}\\
&+ \overline{\Grad(Q_{\de,\om}^{m}+D_{\de,\om}^{m}+\delta (Q_{\de,\om}P_{\de,\om}))}= \mu_{\om} \Delta \vc{v}_{\de,\om}
 -\frac{\mu_{\om}}{K} \vc{v}_{\de,\om}, \label{pressure3}
\end{split}
\end{equation}
and considering the test functions
$$\varphi_i(t,x) = \psi(t) \phi(x) {\mathcal A}_i[\overline{T_k(P_{\de,\om})}], \,\, \psi\in {\mathcal D}(0,T),\,\, \phi \in {\mathcal D}(\mathcal{K}), \,\, i=1,2,3,$$
to deduce
 \[\int_0^T\!\! \int_{\mathcal{K}} \psi \phi 
 \left\{\overline{P^m_{\de,\om} + \de P^\beta_{\de,\om}}  -  \mu_{\om} \Div_x \vc{v}_{\de,\om} \right\} \overline{T_k(P_{\de,\om})} ~dxdt =\]
  \[\int_0^T\!\! \int_{\mathcal{K}} \psi \phi 
 \left\{\overline{Q^m_{\de,\om} +D^m_{\de,\om} + \de(Q^\beta_{\de,\om} +D^\beta_{\de,\om})}  -  \mu_{\om} \Div_x \vc{v}_{\de,\om} \right\} \overline{T_k[P_{\de,\om}]} ~dxdt \]
  \[- \int_0^T \!\!\int_{\mathcal{K}}  \psi \overline{P^m_{\de,\om} +Q^m_{\de,\om} +D^m_{\de,\om} + \de(P^\beta_{\de,\om} +Q^\beta_{\de,\om} +D^\beta_{\de,\om})}\partial_{x_i} \phi {\mathcal A}_i[\overline{T_k(P_{\de,\om})}]~dxdt\]
 \[-\frac{\mu_{\om}}{K} \int_0^T\!\! \int_{\mathcal{K}}  \vc{v} \psi(t) \phi(x){\mathcal A}_i[\overline{T_k(P_{\de,\om})}] ~dxdt+ \mu_{\om} \int_0^T \!\!\int_{\mathcal{K}}\psi \partial_{x_j}\phi \partial_{x_j}\vc{v}^i_{\de,\om} {\mathcal A}_i[\overline{T_k(P_{\de,\om})}] ~dxdt\]
 \[- \mu_{\om} \int_0^T\!\! \int_{\mathcal{K}}\psi \left\{ \vc{v}^i_{\de,\om}\partial_{x_j}\phi \partial_{x_j} {\mathcal A}_i[\overline{T_k(P_{\de,\om})}] + \vc{v}^i_{\de,\om} \partial_{x_i} \phi \overline{T_k(P_{\de,\om})}\right\}~dxdt
\]
 \[
 -\int_0^T\!\! \int_{\mathcal{K}}\phi \vr_{\de,\om}  \vc{v}^i_{\de,\om}  \partial_t \psi {\mathcal A}_i[\overline{T_k(\vr)}] ~ dxdt
 \]
 \[
 - \int_0^T\!\! \int_{\mathcal{K}}\phi  \vr_{\de,\om}  \vc{v}^i  \psi {\mathcal A}_i [\overline{(T'_{k}(P_{\delta,\omega})P_{\delta,\omega}-T_{k}(P_{\delta,\omega}))\Div\vc{v}_{\delta,\omega}}]~dxdt \]
 \[+ \int_0^T\!\! \int_{\mathcal{K}}\phi\vr_{\de,\om}  \vc{v}^i_{\om,\de} \psi   {\mathcal A}_i[\overline{T'_{k}(P_{\delta,\omega})P_{\delta,\omega}}F(C_{\delta,\omega})]~dxdt
\]
\[
 -\int_0^T\!\! \int_{\mathcal{K}}\psi \overline{\vr}_{\de,\om} \vc{v}^i_{\de,\om}    \vc{v}^j_{\de,\om} \partial_{x_j} \phi {\mathcal A}_i[\overline{T_k(P_{\de,\om})}]~dxdt
 \]
\begin{equation}
+ \int_0^T\!\! \int_{\mathcal{K}}\psi  \vc{v}^i _{\de,\om}  \left\{\overline{T_k(P_{\de,\om})} {\mathcal R}_{i,j}[\phi \vr \vc{v}^j _{\de,\om}] - \phi \vr\vc{v}^j {\mathcal R}_{i,j}[\overline{T_k(P_{\de,\om})}]\right\}~dxdt
\label{wc2}
 \end{equation}
 The following result
\[
 \int_0^T\!\! \int_{\mathcal{K}}\psi  \vc{v}_{\de,\om}^i   \left\{T_k(P_{\de,\om,\e}) {\mathcal R}_{i,j}[\phi \vr_{\de,\om,\e}\vc{v}_{\de,\om,\e}^j ] - \phi \vr_{\de,\om} \vc{v}^j_{\de, \om} {\mathcal R}_{i,j}[T_k(P_{\de,\om})]\right\}~dxdt 
\]
$$ \downarrow$$
\[
 \int_0^T\!\! \int_{\mathcal{K}}\psi  \vc{v}^i   \left\{\overline{T_k(P_{\de,\om})} {\mathcal R}_{i,j}[\phi \vr_{\de,\om} \vc{v}^j_{\de,\om} ] - \phi \vr_{\de,\om} \vc{v}^j_{\de,\om} {\mathcal R}_{i,j}[\overline{T_k(P_{\de,\om})}]\right\}~dxdt
\]
the proof of which follows the analysis presented in \cite{Feireisl-book} combined with the analysis performed in \eqref{cPQDC}-\eqref{conv} yields that all the terms on the right-hand side of \eqref{wc1} converge to their counterparts in \eqref{wc2}  ending up with \eqref{evpP}.
The proofs of \eqref{evpQ}, \eqref{evpD} follow in a similar way combining together \eqref{TQe}, \eqref{TQ}, \eqref{pressure3} and  \eqref{TDe}, \eqref{TD}, \eqref{pressure3} respectively.
\end{proof}

\subsubsection{The amplitude of oscillations}
The main result of this section follows the analysis in \cite{Feireisl-book}. Here we only give a flavor of the argument.
\begin{proposition}\label{P6.2}
There exists a constant $c$ independent of $k$  such that 
\begin{equation}
\limsup_{\e \to 0} \|T_k(P_{\de,\om,\e}) - T_k(P_{\de,\om}) \|_{L^{m +1}((0,T) \times B)} \le c.
\label{aoP}
\end{equation}
\begin{equation}
\limsup_{\e \to 0} \|T_k(Q_{\de,\om,\e}) - T_k(Q_{\de,\om}) \|_{L^{m +1}((0,T) \times B)} \le c.
\label{aoQ}
\end{equation}
\begin{equation}
\limsup_{\e \to 0} \|T_k(D_{\de,\om,\e}) - T_k(D_{\de,\om}) \|_{L^{m +1}((0,T) \times B)} \le c.
\label{aoD}
\end{equation}
for any $k \ge 1.$
\end{proposition}

\begin{proof}
We start by proving \eqref{aoP}. For any compact  $\mathcal{K} \subset [0,T]\times\overline{B}$, such that $\displaystyle{\mathcal{K}\cap\Big(\bigcup_{\tau\in[0,T]}(\{\tau\}\times \Gamma_{\tau})\Big)=\emptyset}$ we have 

$$\lim_{\e\to 0} \int_0^T \!\!\int_{\mathcal{K}} \Big(P_{\de,\om,\e}^m +\delta P_{\de,\om,\e}^\delta) T_k(P_{\de,\om,\e})-\overline{P_{\de,\om}^m +\delta P_{\de,\om}^\delta}~\overline{T_k(P_{\de,\om})}\Big)dxdt=$$
$$\lim_{\e\to 0} \int_0^T \!\!\int_{\mathcal{K}}   \Big(P_{\de,\om,\e}^m T_k(P_{\de,\om,\e})-\overline{P_{\de,\om}^m} ~\overline{T_k(P_{\de,\om})}\Big)dxdt$$
\begin{equation}
+\delta\lim_{\e\to 0} \int_0^T\!\!\int_{\mathcal{K}}  \Big( P_{\de,\om,\e}^\delta T_k(P_{\de,\om,\e})-\overline{P_{\de,\om}^\delta}~\overline{T_k(P_{\de,\om})}\Big)dxdt,
\label{aoP1}
\end{equation}
where, by using the convexity of the function $z \to z^{\gamma}$ is convex, and the fact that $T_k(z)$ is concave on $[0,\infty)$  we can prove
\begin{equation}
\begin{split}
\lim_{\e\to 0} \int_0^T\!\! \int_{\mathcal{K}}  P^m_{\de,\om,\e} T_k(P_{\de,\om,\e}) - \overline{P^m_{\de,\om}}\, \overline{T_k(P_{\de,\om})} ~dxdt\geq\\
\limsup_{\e\to 0} \int_0^T\!\! \int_{\mathcal{K}}  |T_k(P_{\de,\om}) - T_k(P)|^{m+1} ~dxdt.
\end{split}
\label{aoP2}
\end{equation}
Since $z \to z^{\beta}$ and  $T_k(z)$ are non decreasing  we have that,
\begin{equation}
\delta\lim_{\e\to 0} \int_0^T\!\!\int_{\mathcal{K}}  \Big( P_{\de,\om,\e}^\delta T_k(P_{\de,\om,\e})-\overline{P_{\de,\om}^\delta}~\overline{T_k(P_{\de,\om})}\Big)dxdt\geq 0.
\label{aoP3}
\end{equation} 
On the other hand,
$$\lim_{\e \to 0} \int_0^T\!\! \int_{\mathcal{K}}  \left( \Div_x \vc{v}_{\de,\om,\e} \, T_k(P_{\de,\om,\e}) - \Div_x \vc{v}_{\de,\om}\, \overline{T_k(P_{\de,\om})}\right) ~dxdt=$$
$$\lim_{\e\to 0} \int_0^T\!\! \int_{\mathcal{K}}  \left(T_k(P_{\de,\om,\e}) - T_k(P_{\de,\om}) + T_k(P_{\de,\om}) - \overline{T_k(P_{\de,\om})}\right) \Div_x \vc{v}_{\de,\om,\e} ~dxdt \le $$
 $$ 2 \sup_{\e}  \|\Div_x \vc{v}_{\de,\om,\e}\|_{L^2((0,T) \times {\mathcal{K}} )} \limsup_{\e \to 0}  \|T_k(P_{\de,\om,\e})- T_k(P_{\de,\e})\|_{L^{2}((0,T) \times {\mathcal{K}} )}.$$
By combing \eqref{aoP1}, \eqref{aoP2}, \eqref{aoP3}  together with \eqref{evpP} we have that
$$\limsup_{\e \to 0} \|T_k(P_{\de,\om,\e}) - T_k(P_{\de,\om}) \|_{L^{m +1}((0,T) \times \mathcal{K})} \le c.
$$
The result  \eqref{aoP} now follows since the constant $c$ is independent of $\mathcal{K}$.
The cell densities $Q$ and $D$ can be treated in a similar fashion yielding \eqref{aoQ} and \eqref{aoD}.

\end{proof}

\subsubsection{On the oscillations defect measure} 
In this section we will perform the final step of the proof of the strong convergence of our densities.
For simplicity we start with the density of the proliferating cells.
If we denote by $S_{\e}$ a regularizing operator and apply it to the equation \eqref{TP} we get 
\begin{equation}
\begin{split}
\partial_{t}S_{\e}[ \overline{T_{k}(P_{\delta,\omega})}] &+\Div(S_{\e}[\overline{T_{k}(P_{\delta,\omega})}]\vc{v}_{\delta,\omega})+S_{\e}[\overline{(T'_{k}(P_{\delta,\omega})P_{\delta,\omega}-T_{k}(P_{\delta,\omega}))\Div\vc{v}_{\delta,\omega}}]\\&=r_{\e}+S_{\e}[\overline{T'_{k}(P_{\delta,\omega})P_{\delta,\omega}}F(C_{\delta,\omega})],
\end{split}
\label{TPR1}
\end{equation}
where $r_{\e}\to 0$ in $L^{2}(0,T;L^{2}(\R^{3}))$ for any fixed $k$.
Multiplying \eqref{TPR1} by $b'(S_{\e}[ \overline{T_{k}(P_{\delta,\omega})}])$ and letting $\e\to 0$ we deduce
\begin{equation}
\begin{split}
\partial_{t}b(\overline{T_{k}(P_{\delta,\omega})})+\Div(b(\overline{T_{k}(P_{\delta,\omega})})\vc{v}_{\delta,\omega})\\
+\left(b'(\overline{T_{k}(P_{\delta,\omega})})\overline{T_{k}(P_{\delta,\omega})}-b(\overline{T_{k}(P_{\delta,\omega})})\right)\Div\vc{v}_{\delta,\omega}\\
=b'(\overline{T_{k}(P_{\delta,\omega})})[\overline{(T_{k}(P_{\delta,\omega})-T'_{k}(P_{\delta,\omega})P_{\de,\om})\Div\vc{v}_{\delta,\omega}}]\\
+b'(\overline{T_{k}(P_{\delta,\omega})})\overline{T'_{k}(P_{\delta,\omega})P_{\delta,\omega}}F(C_{\delta,\omega})].
\end{split}
\label{TPR2}
\end{equation}
Now we send $k\to\infty$ in \eqref{TPR2} and follow the same line of arguments as in \cite{Feireisl-book} and we end up with
\begin{equation}
\begin{split}
\partial_{t}b(P_{\delta,\omega})+\Div(b(P_{\delta,\omega})\vc{v}_{\de,\om})&+(b'(P_{\delta,\omega})P_{\delta,\omega}-b(P_{\delta,\omega}))\Div\vc{v}_{\de,\om}\\
&=b'(P_{\delta,\omega})P_{\delta,\omega}F(C_{\delta,\omega}).
\end{split}
\label{TPR3}
\end{equation}
Let us introduce a family of functions $L_k$ as,
\begin{equation}\label{cut2}
 L_k(z)= 
\begin{cases}
z \log(z) \,\, \mbox{for}\,\, 0\le z< k,\\
z \log(k) + z \int_k^z \frac{T_k(s)}{s^2} ds \,\,\mbox{for}\,\, z\ge k.
\end{cases}
\end{equation}
Seeing that $L_k$ can be written as
$$L_k(z) = \beta_k z + b_k(z)$$
where $|b_k(z)| \le c(k)$ and $b_k'(z)z - b_k(z) = T_k(z)$ for all $z>0,$ by considering \eqref{dP} we obtain \begin{equation}
\begin{split}
\partial_{t} L_{k}(P_{\delta,\omega,\e}) &+\Div(L_{k}(P_{\delta,\omega,\e}) \vc{v}_{\delta,\omega,\e})+T_{k}(P_{\delta,\omega,\e}) \Div\vc{v}_{\delta,\omega,\e}\\&=L'_{k}(P_{\delta,\omega,\e})P_{\delta,\omega,\e}F(C_{\delta,\omega,\e}).
\end{split}
\label{DM1}
\end{equation}
and by virtue of \eqref{TPR3}  we arrive at 
\begin{equation}
\hspace{-0,2cm}\partial_{t} L_{k}(P_{\de,\om}) +\Div(L_{k}(P_{\de,\om}) \vc{v})+T_{k}(P_{\de,\om}) \Div\vc{v}_{\de,\om}=L'_{k}(P_{\de,\om})P_{\de,\om}F(C_{\de,\om})\label{DM2}
\end{equation}
in ${\mathcal D}'((0,T) \times B)$.

Consequently, we can assume 
$$ L_k(P_{\de,\om,\e}) \to \overline{L_k(P_{\de,\om})} \,\, \mbox{in} \,\, C([0,T];L^m_{weak}(B))$$
and approximating $z \log(z) \approx L_k(z)$,
$$P_{\de,\om,\e} \log(P_{\de,\om,\e}) \to \overline{P_{\de,\om} \log(P_{\de,\om})} \,\,\mbox{in} \,\, C([0,T]; L^{\alpha}_{weak}(B))$$
for any $1 \le \alpha < m.$
Taking the difference of \eqref{DM1} and \eqref{DM2}, integrating with respect to $t$  and by using the conditions \eqref{IP} on the initial data and the boundary conditions \eqref{no-slip} we get
\begin{equation}
\begin{split}
\int_{B}\Big(L_k(P_{\de,\om,\e})&-L_{k}(P_{\de,\om})\Big)dx\\
=\int_{0}^{t}\!\!\int_{B}\Big(T_{k}(P_{\de,\om}) \Div\vc{v}_{\de,\om}-&T_{k}(P_{\delta,\omega,\e}) \Div\vc{v}_{\delta,\omega,\e}\Big)dxdt\\
+\int_{0}^{t}\!\!\int_{B}\Big(L'_{k}(P_{\delta,\omega,\e})P_{\delta,\omega,\e}F(C_{\delta,\omega,\e})&-L'_{k}(P_{\de,\om})P_{\de,\om}F(C_{\de,\om})\Big)dxdt.
\end{split}
\label{DM3}
\end{equation}
By combining the monotonicity of $z\to L'_{k}(z)z$ and of the pressure with the Proposition \ref{P1.1} we pass into the limit  for $\e\to 0$ in \eqref{DM3} and, in the spirit of the analysis in \cite{Feireisl-book}, we deduce 
\begin{equation}
\begin{split}
\int_{B}(\overline{L_{k}(P_{\de,\om})}&-L_{k}(P_{\de,\om}))(t)dx\\
\leq\int_{0}^{t}\!\!\int_{B}(T_{k}(P_{\de,\om})&-\overline{T_{k}(P_{\de,\om})})\Div\vc{v}_{\de,\om}dxdt.
\end{split}
\label{DM4}
\end{equation}
By virtue of the Proposition \ref{P6.2}, the right-hand side  of \eqref{DM4} tends to zero as $k \to \infty.$
Accordingly, passing to the limit for $k\to \infty$ we conclude that 
$$\overline{P_{\de,\om} \log(P_{\de,\om})}(t) = P_{\de,\om} \log(P_{\de,\om})(t)\,\, \mbox{for all}\,\, t\in [0,T].$$
 which implies
 $$P_{\de,\om,\e}\longrightarrow P_{\de,\om,\e}\quad \text{a.e. in $(0,T)\times B$}.$$
 One can treat in a similar fashion the densities $Q_{\de,\om,\e}$ and $D_{\de,\om,\e}$ to conclude
 $$Q_{\de,\om,\e}\longrightarrow Q_{\de,\om,\e}\quad \text{a.e. in $(0,T)\times B$},$$
 $$D_{\de,\om,\e}\longrightarrow D_{\de,\om,\e}\quad \text{a.e. in $(0,T)\times B$}.$$

 \subsection{Vanishing density terms  in the ``healthy tissue''}
By using the strong convergence of the previous section, the momentum equation \eqref{w-pressure22} now reads as follows
\begin{equation} \nonumber
\int_{B}\vr_{\delta,\omega}\vc{v}_{\delta,\omega}\vc{\varphi}(\tau,\cdot)dx- \int_{B}{ (\vr\vc{v})_0 \cdot \vc{\varphi} (0, \cdot)\,\dx}
\end{equation}
\begin{equation*}
= \int_0^{\tau}\!\! \int_{B} { \Big( \vr_{\delta,\omega} \vc{v}_{\delta,\omega} \cdot
\partial_t \vc{\varphi} + \vr \vc{v}_{\delta, \omega,} \otimes \vc{v}_{ \delta, \omega} : \Grad_x \vc{\varphi} + \sigma_{\delta} (P_{\delta, \omega},Q_{\delta, \omega},D_{\delta, \omega}) \ \Div \vc{\varphi} \Big)}dxdt 
\end{equation*}
\begin{equation}
 -\int_0^\tau\!\! \int_{B} \left(\mu_{\omega}
 \Grad_x \vc{v}_{\delta, \omega}:\Grad_x \vc{\varphi} + \frac{{\mu}_{\omega}} {K} \vc{v}_{\delta,\omega} \vc{\varphi}\right)\,  \dxdt 
\label{w-pressure222}
\end{equation}
for any test functions as in \eqref{w-pressure22}.

The next issue now is to get rid of the density terms supported in the healthy tissue part $((0,T)\times B)\backslash Q_{T}$.  In order to achieve this aim one has to  describe the evolution of the interface $\Gamma_{\tau}.$ To that effect we employ elements from the so-called {\em level set method}.
The {\em level set method} is a numerical method for tracking interfaces and shapes (cf.\ Osher and Fedwik \cite{OshFed03}). It turns out that the interface $\Gamma_\tau$ can be identified with a component of the  set 
$$\{ \Phi(\tau, \cdot)=0\},$$
while the set $B\setminus \Omega_{\tau}$ correspond to $\{\Phi(\tau,\cdot)>0\}$, with
 $\Phi=\Phi(t,x)$ denoting the  unique solution of the transport equation 
\begin{equation}
\partial_{t}\Phi+\Grad_{x}\Phi(t,x)\cdot\vc{V}=0,
\label{dd1}
\end{equation}
with initial data
$$\Phi_{0}(x)=
\begin{cases}
>0 & \text{for}\  x\in B\backslash\Omega_{0},\\
<0 & \text{for}\  x\in \Omega_{0}\cup(\R^{3}\backslash\overline{B}),
\end{cases}
\qquad \Grad_{x}\Phi_{0}\neq 0\  \text{on $\Gamma_{0}$}.$$
Finally,
\begin{equation}
\begin{split}
\Grad_{x}\Phi(\tau,x)&=\lambda(\tau,x)\vc{n}(x) \qquad \text{for any $x\in \Gamma_{\tau}$}\\
&\\
&\lambda(\tau,x)\geq 0 \qquad \text{for $\tau\in[0,T]$.}
\end{split}
\label{d1}
\end{equation}
In order to estimate the behavior of our approximating sequences on the healthy tissue we need to prove the following lemma.
\begin{lemma}
Let $Z\in L^{\infty}(0,T;L^{2}(B))$, $Z\geq 0$, $\vc{v}\in W^{1,2}_{0}(B)$ satisfying the following equation
\begin{equation} 
\begin{split}
 \int_{B}  \big( Z \varphi(\tau,\cdot)&- Z_0 \varphi(0,\cdot)\big)dx\\
 & = \int_0^{\tau}\! \int_{B} \left( Z \partial_t \varphi + Z \vc{v} \cdot \Grad_x \varphi + \vc{G_Z}  \varphi(t, \cdot) \right) dx dt,
\end{split}
\label{Zeq}
\end{equation}
for any $\tau\in[0,T]$ and any test function $\varphi\in C^{1}_{c}([0,T]\times \R^{3})$  and $\vc{G_Z}$ a linear function of $Z$. Moreover assume that
\begin{equation}
(\vc{v}-\vc{V})(\tau, \cdot)\cdot\vc{n}\big|_{\Gamma_{\tau}}=0\qquad \text{a.e. $\tau\in (0,T)$}
\label{ueq}
\end{equation}
and that
$$Z_{0}\in L^{2}(\R^{3}), \qquad Z_{0}\geq 0\qquad Z_{0}\big|_{B\backslash\Omega_{0}}=0.$$
Then
\begin{equation*}
Z(\tau, \cdot)\big|_{B\backslash\Omega_{\tau}}=0 \qquad \text{for any $\tau\in [0,T]$}.
\label{zeroZ}
\end{equation*}
\label{fondlemma}
\end{lemma}
\begin{proof}
For a detailed proof we refer the reader to \cite{DT-MixedModel-2013}. We present here only the main idea for completeness. The strategy relies on the  construction of an appropriate test function in the weak formulation \eqref{Zeq}. For given $\eta>0$ we use
\begin{equation}
\varphi=\left[\min\left\{\frac{1}{\eta}\Phi;1\right\}\right]^{+}
\label{specialtest}
\end{equation}
as a test function in   \eqref{Zeq} and we  obtain
\begin{equation} 
\begin{split}
 \int_{B\setminus \Omega_{\tau}} \!\!Z \varphi \, dx = &\frac{1}{\eta}\int_0^{\tau}\! \int_{\{0\leq \Phi(t,x)\leq \eta\}} \left( Z \partial_t \Phi + Z \vc{v} \cdot \Grad_x \Phi + \vc{G_Z}  \Phi\right) dx dt\\
&+ \int_0^{\tau}\! \int_{\{\Phi(t,x)> \eta\}} \vc{G_Z} dxdt.
\end{split}
\label{Zeq2}
\end{equation}
Observing that
$$Z \partial_t \Phi + Z \vc{v} \cdot \Grad_x \Phi=Z(\vc{v}-\vc{V}) \cdot \Grad_x \Phi$$
and using  \eqref{d1} and \eqref{ueq}   we get
\begin{equation}
(\vc{v}-\vc{V}) \cdot \Grad_x \Phi\in W^{1,2}_{0}(B\backslash\Omega_{\tau}) \quad \text {for a.e. $t\in (0,\tau)$.}
\label{d2}
\end{equation}
Introducing the  distance function $\delta=\delta(t,x)$ of the form
\begin{equation}
\delta(t,x)=dist_{\R^{3}}[x, \partial(B\backslash\Omega_{\tau})] \qquad \text{for $t\in [0,\tau]$, $x\in B\backslash\Omega_{\tau}$},
\label{dist}
\end{equation}
relation \eqref{d2} yields
\begin{equation}
\frac{1}{\delta}(\vc{V}-\vc{v}) \cdot \Grad_x \Phi \in L^{2}([0,\tau]\times B\backslash\Omega_{\tau} ).
\label{Hardy}
\end{equation}
Using \eqref{Zeq2}, \eqref{Hardy}, the regularity of  $\vc{V}$ and letting $\eta \to 0$  in \eqref{Zeq2} (noting that  $\vc{G_{Z}}$ is a linear function of $Z$ with  $Z\in L^{\infty}(0,T;L^{2}(B))$) 
we obtain the result. 
\end{proof}
Now we are ready to prove that the proliferating, quiescent, dead cells and the nutrient are vanishing in the healthy tissue. 
\begin{proposition}
Assume that $P_{\delta,\omega}$, $Q_{\delta,\omega}$, $D_{\delta,\omega}$ and $C_{\omega}$ are solutions of \eqref{p-D} and that \eqref{IP} holds, then
\begin{equation}
 P_{\delta,\omega}(\tau, \cdot)|_{B\backslash\Omega_{\tau}}=0, \quad  Q_{\delta,\omega}(\tau, \cdot)|_{B\backslash\Omega_{\tau}}=0, \quad  D_{\delta,\omega}(\tau, \cdot)|_{B\backslash\Omega_{\tau}}=0.
 \label{sdPQD}
\end{equation}
\begin{equation}
C_{\delta,\omega}(\tau, \cdot)|_{B\backslash\Omega_{\tau}}=0.
 \label{sdC}
\end{equation}
\label{fondprop}
\end{proposition}
\begin{proof}
The proof of \eqref{sdPQD} follows applying the  Lemma \ref{fondlemma}.  In fact since $\beta\geq 2$ from \eqref{bP}, \eqref{bQ},\eqref{bD} we have that  $P_{\delta,\omega}$, $Q_{\delta,\omega}$, $D_{\delta,\omega}$ are bounded in $L^{\infty}(0,T;L^{2}(B))$, for any fixed $\delta$. Moreover by taking into account \eqref{G} and \eqref{mpC} 
the functions $\vc{G_P}, \vc{G_Q}, \vc{G_D}$ fulfill the requirements of the Lemma \ref{fondlemma}.
In order to prove \eqref{sdC} it is enough to observe that $C_{\delta,\omega}$ is a solution in $B\backslash\Omega_{\tau}$ of a parabolic equation with vanishing initial  and boundary data.
\end{proof}
Now taking into account the Proposition \ref{fondprop} the equation for the nutrients becomes 
\begin{equation*}
\int_{\Omega_{\tau}} C_{\delta,\omega} \varphi(\tau,\cdot) dx -  \int_{\Omega_{0}} C_0 \varphi(0,\cdot) dx
\end{equation*}
\begin{equation}
= \int_0^{\tau}\!\! \int_{\Omega_{t}} \left( C_{\delta,\omega} \partial_t \varphi - \nu_{\omega}\Grad_{x} C_{\delta,\omega} \cdot \Grad_x \varphi -C_{\delta,\omega} \varphi(t, \cdot) \right) dx dt 
\label{pC13}
\end{equation}
for any $\tau \in [0,T]$ and any test function $\vc{\varphi} \in C_c^{\infty}([0,T] \times B; \mathbb{R}^3)$, while the momentum equation \eqref{w-pressure222} becomes as follows,
\begin{equation*} 
\int_{\Omega_{\tau}}\vr_{\delta,\omega}\vc{v}_{\delta,\omega}\vc{\varphi}(\tau,\cdot)dx- \int_{\Omega_{0}}{ (\vr\vc{v})_0 \cdot \vc{\varphi} (0, \cdot)\,\dx}
\end{equation*}
\begin{equation*}
= \int_0^{\tau}\!\! \int_{\Omega_{t}} { \Big( \vr_{\delta,\omega} \vc{v}_{\delta,\omega} \cdot
\partial_t \vc{\varphi} + \vr \vc{v}_{\delta, \omega,} \otimes \vc{v}_{ \delta, \omega} : \Grad_x \vc{\varphi} + \sigma_{\delta} (P_{\delta, \omega},Q_{\delta, \omega},D_{\delta, \omega}) \ \Div \vc{\varphi} \Big)}dxdt 
\end{equation*}
\begin{equation*}
 -\int_{0}^{\tau}\!\! \int_{\Omega_{t}} \left(\mu_{\omega}
 \Grad_x \vc{v}_{\delta, \omega}:\Grad_x \vc{\varphi} + \frac{{\mu}_{\omega}} {K} \vc{v}_{\delta,\omega} \vc{\varphi}\right)\,  \dxdt 
\end{equation*}
\begin{equation}
 -\int_{0}^{\tau}\!\! \int_{B\backslash\Omega_{t}} \left(\mu_{\omega}
 \Grad_x \vc{v}_{\delta, \omega}:\Grad_x \vc{\varphi} + \frac{{\mu}_{\omega}} {K} \vc{v}_{\delta,\omega} \vc{\varphi}\right)\,  \dxdt 
\label{w-pressure2222}
\end{equation}

\section{Vanishing viscosity limit $\omega \to 0$}\label{S7}
The next step in the proof is to get rid of the last integrals in  \eqref{w-pressure2222}, so we have to perform the limit $\omega \to 0$. By using  \eqref{bv} we have that
\begin{equation}
\begin{split}
\int_{\Omega_{t}}\mu\left(|\nabla_{x}\vc{v}_{\omega}|^{2}+|\vc{v}_{\omega}|^{2}\right)dx\leq c\\
\int_{B\backslash\Omega_{t}}\mu_{\omega}\left(|\nabla_{x}\vc{v}_{\omega}|^{2}+|\vc{v}_{\omega}|^{2}\right)dx\leq c,
\end{split}
\label{visc2}
\end{equation}
%\begin{equation}
%\int_{0}^{T}\!\!\int_{\Omega_{t}}\nu\left(|\nabla_{x}C_{\omega}|^{2}\right)dxdt\leq c,
%\nonumber
%\end{equation}
%\begin{equation}
%\int_{0}^{T}\!\!\int_{B\backslash\Omega_{t}}\nu_{\omega}\left(|\nabla_{x}C_{\omega}|^{2}\right)dxdt\leq c.
%\label{visc4}
%\end{equation}
The estimates \eqref{visc2}  with a standard computations yields that
\begin{equation}
\int_{B\backslash\Omega_{t}}\mu_{\omega}\left(\Grad_x \vc{v}_{\omega}  : \Grad_x \vc{\varphi}  +\vc{v}_{\omega,\e} \vc{\varphi}\right)dx\to 0 
\quad \text{as $\omega\to 0$},
\label{visc5}
\end{equation}
%\begin{equation}
%\int_{0}^{\tau}\!\!\int_{B\backslash\Omega_{t}}\nu_{\omega}\left(\Grad_x C_{\omega,\e}  : \Grad_x \vc{\varphi}  \right)dxdt \to 0
%\quad \text{as $\omega\to 0$}.
%\label{visc6}
%\end{equation}
Now, by repeating the same arguments of the previous sections and taking into account that now we only need the compactness of the densities only in the tumor region we let $\omega\to 0$ and we get that the nutrient has the form
\begin{equation*}
\int_{\Omega_{\tau}} C_{\delta} \varphi(\tau,\cdot) dx -  \int_{\Omega_{0}} C_0 \varphi(0,\cdot) dx
\end{equation*}
\begin{equation}
= \int_0^{\tau}\!\! \int_{\Omega_{t}} \left( C_{\delta} \partial_t \varphi - \nu\Grad_{x} C_{\delta} \cdot \Grad_x \varphi - C_{\delta} \varphi(t, \cdot) \right) dx dt .
\label{pC11}
\end{equation}
The momentum equation \eqref{w-pressure2222} is now the following,

\begin{equation*} 
\int_{\Omega_{\tau}}\vr_{\delta}\vc{v}_{\delta}\vc{\varphi}(\tau,\cdot)dx- \int_{\Omega_{0}}{(\vr \vc{v})_0 \cdot \vc{\varphi} (0, \cdot)\,\dx}
\end{equation*}
\begin{equation*}
= \int_0^{\tau}\!\! \int_{\Omega_{t}} { \Big( \vr_{\delta} \vc{v}_{\delta} \cdot
\partial_t \vc{\varphi} + \vr \vc{v}_{\delta} \otimes \vc{v}_{\delta} : \Grad_x \vc{\varphi} + \sigma_{\delta} (P_{\delta},Q_{\delta},D_{\delta}) \ \Div \vc{\varphi} \Big)}dxdt 
\end{equation*}
\begin{equation}
 -\int_{0}^{\tau}\!\! \int_{\Omega_{t}} \left(\mu
 \Grad_x \vc{v}_{\delta}:\Grad_x \vc{\varphi} + \frac{{\mu}} {K} \vc{v}_{\delta} \vc{\varphi}\right)\,  \dxdt. 
\label{w-pressure22222}
\end{equation}

\section{Vanishing artificial pressure $\delta \to 0$}\label{S8}
Finally, the last step in our proof is to get rid of the artificial pressure term $\delta(P^{\beta}_{\delta}+Q^{\beta}_{\delta}+D^{\beta}_{\delta})$.  In order to pass into the limit we need the strong convergence of the cell densities. The main part consists in showing  that the oscillation defect  measure 
$$
\sup_{k\geq 0}\left(\limsup_{\de \to 0}\|T_k(Z_{\de}) - T_k(Z) \|_{L^{m +1}((0,T) \times \Omega)} \right)
$$

is bounded. This can be done in the same spirit of the earlier section (see also \cite{Donatelli-Trivisa-2006}). Now, we are ready to let $\delta\to 0$ in the weak formulations \eqref{pC13},  \eqref{w-pressure22222} and we complete the proof of Theorem  \ref{T2.2}.
\section{Acknowledgments}
The work of D.D. was supported  by the Ministry of Education, University and Research (MIUR), Italy under the grant PRIN 2012- Project N. 2012L5WXHJ, \emph{Nonlinear Hyperbolic Partial Differential Equations, Dispersive and Transport Equations: theoretical and applicative aspects}. Ê K.T.  gratefully acknowledges the support  in part by the National Science Foundation under the grant DMS-1211519 and by the Simons Foundation under the Simons Fellows in Mathematics Award 267399. 
Part of this research was performed during the visit of K.T. at University of L'Aquila which was supported under the grant PRIN 2012- Project N. 2012L5WXHJ, \emph{Nonlinear Hyperbolic Partial Differential Equations, Dispersive and Transport Equations: theoretical and applicative aspects}.
This work was completed while K.T. was resident at \'{E}cole Normale Sup\'{e}rieure de  Cachan as a Simons Fellow. K.T. is grateful to L.\ Desvilettes and the CMLA Lab for providing a very stimulating environment for scientific research  and to the Institute Henri Poincar\'{e} for the hospitality.

\end{document}